\documentclass[12pt]{amsart}
\usepackage{setspace}
\usepackage[T1]{fontenc}
\usepackage{tgtermes}
\usepackage{amsmath, amsthm, amsfonts, amssymb}
\usepackage{mathrsfs}
\usepackage{graphicx}
\usepackage{hyperref}
\usepackage{url}
\usepackage[all]{xy}
\usepackage{enumerate}
\usepackage{fancyhdr}
\usepackage{xspace}
\usepackage{verbatim}
\usepackage{fullpage}%
\usepackage{subfloat}
\usepackage[a4paper, landscape]{lscape}
\usepackage{subfig}

\usepackage{tikz,xcolor,color}
\usetikzlibrary{calc,arrows,decorations.pathmorphing,backgrounds,fit,automata,shapes.geometric,shapes.arrows}
\usetikzlibrary{positioning}
\theoremstyle{definition}
\newtheorem{defn}{Definition}[section]

\theoremstyle{remark}

\theoremstyle{plain}

\newtheorem{conj}[defn]{Conjecture}

\newtheorem{lem}[defn]{Lemma}

\newtheorem{prop}[defn]{Proposition}

\newtheorem{thm}[defn]{Theorem}

\numberwithin{table}{section}
\newtheorem{rem}[defn]{Remark}

\newtheorem*{theorem*}{Theorem}
\newtheorem*{proposition*}{Proposition}

\numberwithin{equation}{section}
\numberwithin{figure}{section}

\newcommand{\set}[1]{\left\{#1\right\}}

\newcommand{\round}[1]{{\ooalign{\hfil\raise .10ex\hbox{\scriptsize#1}\hfil\crcr\mathhexbox20D}}}

\newcommand{\hex}{K}
\newcommand{\alb}{\mathsf{X}}

\newcommand{\sss}{\mathcal{L}}

\newcommand{\shift}{\sigma}

\newcommand{\face}{\mathsf{G}}
\newcommand{\edge}{\mathsf{E}}
\newcommand{\shiftspace}{\Sigma}
\newcommand{\leftshift}{\sigma}

\long\def\red#1{\textcolor[rgb]{1.00,0.00,0.00}{#1}}
\long\def\red#1{}

\def\s-s{self-similar}

\allowdisplaybreaks[4]
\title{Random walks on barycentric subdivisions and the Strichartz hexacarpet}

\thanks{Research supported in part by NSF grant DMS-0505622}

\author[M.~Begue]{Matthew Begue}
\address[M.~Begue]{Department of Mathematics, University of Connecticut, and Department of Mathematics, University of Maryland, College Park, MD 20742-4015,
USA}
\email[M.~Begue]{matthew.begue@uconn.edu}

\author[D.~J.~Kelleher]{Daniel J.~Kelleher}

\email[D.~J.~Kelleher]{kelleher@math.uconn.edu}
\urladdr{\url{http://www.math.uconn.edu/~kelleher/}}

\author[A.~Nelson]{Aaron Nelson}
\email[A.~Nelson]{aaron.nelson@uconn.edu}

\author[H.~Panzo]{Hugo Panzo}
\email[H.~Panzo]{panzo@math.uconn.edu}
\urladdr{\url{http://www.math.uconn.edu/~panzo/}}

\author[R.~Pellico]{Ryan Pellico}
\email[R.~Pellico]{ryan.pellico@uconn.edu}
\urladdr{\url{http://www.math.uconn.edu/~pellico/}}

\author[A.~Teplyaev]{Alexander Teplyaev}
\email[A.~Teplyaev]{teplyaev@uconn.edu}
\urladdr{\url{http://www.math.uconn.edu/~teplyaev/}}

\address{Department of Mathematics, University of Connecticut, Storrs, CT 06269, USA}

\begin{document}
%
%

\begin{abstract}
We investigate simple random walks on graphs generated by 
repeated  barycentric subdivisions of a triangle. We use these random walks to study the diffusion on the self-similar fractal, the Strichartz hexacarpet, which is generated as the limit space of these graphs. 
We make this connection rigorous by establishing a graph  
isomorphism between the hexacarpet approximations and graphs produced by repeated Barycentric subdivisions of the triangle. This includes a discussion of various numerical calculations performed on the these graphs, and their implications to the diffusion on the limiting space. 
In particular, we prove that   equilateral barycentric subdivisions --- a metric space generated by replacing the metric on each 2-simplex of the subdivided triangle with that of a scaled Euclidean equilateral triangle --- converge to a self-similar 
geodesic metric space of dimension $\log(6)/\log(2)$, or about~2.58. 
Our numerical experiments give evidence to a conjecture that   the simple random walks 
on  the equilateral barycentric subdivisions 
converge to 
a continuous diffusion process on the  Strichartz hexacarpet corresponding to  
a different spectral dimension (estimated numerically to be about 1.74). 

\tableofcontents
\end{abstract}\maketitle

\section{Introduction and main conjectures\red{\ (Dan, Sasha et al.)}}

The goal of this paper is to investigate the relation between  simple random walks on  
repeated  barycentric subdivisions of 
a triangle  and the self-similar fractal Strichartz hexacarpet. 
We  explore a graph approximation to the hexacarpet in order to establish a graph  
isomorphism between the hexacarpet and Barycentric subdivisions of the triangle. 
After that we  discuss various numerical calculations performed on the approximating graphs. 
We prove that  the equilateral barycentric subdivisions converge to a self-similar geodesic metric space of dimension $log(6)/log(2)\approx2.58$ but, at the same time, our mathematical experiments give evidence to a conjecture that   the simple random walks converge to 
a continuous diffusion process on the  Strichartz hexacarpet corresponding to the spectral dimension $\approx1.74$. 

In Section \ref{sec:bcs} we develop the framework and basic results pertaining to barycentric subdivision. This is a standard object, intrinsic to the study of simplicial complexes, see \cite{Hat02} (and such classics as \cite{Lefschetz,Pontryagin,Spanier}). We define a metric on the set of edges of $n$th-iterated barycentric subdivision of a $2$-simplex, and use this to define a new limiting self-similar metric on the standard  Euclidean equilateral triangle.

In Section \ref{sec:introduction}, we turn to the theory of self-similar structures, as developed in \cite{kig01}. Using this theory we introduce a fractal structure, which we call the Strichartz hexacarpet, or hexacarpet for short. The hexacarpet is not isometrically embeddable into two dimensional Euclidean space, but otherwise resembles other self-similar infinitely ramified fractals with Cantor-set boundaries, such as the octacarpet (which is sometimes referred to as the octagasket, see \cite{outer} and references therein), 
the Laakso spaces (see \cite{RomeoSteinhurst,Steinhurst} and references therein), 
and 
the standard and generalized Sierpinski carpets (see \cite{BBKT} and references therein). 

We draw a connection between the hexacarpet and barycentric subdivisions in Section \ref{sec:graphiso}, in which we prove that the approximating graphs to the hexacarpet are isomorphic to the graphs created by barycentric subdivisions (where the 2-simplexes of the $n$-times subdivided triangle become graph vertices, connected by a graph edge of the simplexes that share a common face).

Section \ref{sec:graph_distances} discusses properties of the approximating graphs of the hexacarpet to contrast and illuminate connections between the hexacarpet and the limiting structure on the triangle defined in Section \ref{sec:bcs}. In particular, we examine the growth properties of the graph distance metric. We prove a proposition which heuristically places the diameter (in the sense of the usual graph distance) of the $n$th level graph as somewhere between $O(2^n)$ and $O(n2^n)$. Our numerical analysis supports a conjecture for the formulas of the diameter and radius of these graphs.

Finally in Section \ref{sec:numerics} we briefly describe numerical analysis of the spectral properties of the approximating graphs to the hexacarpet. Primarily we calculate eigenvalues of the Laplacian matrix for the first 8 levels of the approximating graphs. This allows us to approximate the resistance scaling factor of the hexacarpet, and suggests that there is a limit resistance. We also  plot approximations to the hexacarpet in 2- and 3-dimensional eigenfunction coordinates. 

 The experimental results in Section~\ref{sec:numerics} strongly suggest that the 
simple random walks on the barycentric subdivisions converge to a diffusion process on $K$ (most  efficiently this can be shown by analyzing harmonic functions and  eigenvalues and eigenfunctions of the Laplacian).
Thus, our theoretical and numerical results support the following conjecture, which is  explanationed in  Section~\ref{sec:numerics}. 

\begin{conj}\label{conj-1}We   conjecture that 
\begin{enumerate}
\item on the Strichartz hexacarpet there exists a unique self-similar local regular conservative Dirichlet form $\mathcal E$ with resistance scaling factor $\rho\approx$1.304 and the Laplacian scaling factor $\tau=6\rho$;   this form is a resistance form in the sense of Kigami. 
\item the simple random walks on the repeated  barycentric subdivisions of 
a triangle, with the time renormalized by $\tau^n$, converge to the  diffusion process, which is the continuous symmetric strong Markov process corresponding to the Dirichlet form $\mathcal E$;
\item this diffusion process satisfies the sub-Gaussian heat kernel estimates and   elliptic and parabolic Harnack inequalities, possibly with logarithmic corrections, corresponding to the Hausdorff dimension $\dfrac{log(6)}{log(2)}\approx2.58$ and the spectral dimension $2\dfrac{log(6)}{log(\tau)}\approx1.74$;  
\item the spectrum of the Laplacian  has spectral gaps in the sense of Strichartz;
\item the spectral zeta function has a meromorphic continuation to $\mathbb C$. 
\end{enumerate}
\end{conj}

We note that our data on spectral dimension is not inconsistent with the (random) geometry of the numerical approximations used in the theory of  quantum gravity, according to the work of Ambj\o rn, Jurkiewicz, and Loll (see \cite{AJL}) at small time asymptotics: $d_S=1.8\pm0.25$. This reference as well as \cite{Reuter} use triangulations similar to those in our study to approximate quantum gravity. Therefore one can conclude that, with the present state of  numerical experiments, fractal carpets may represent a plausible (although simplified) model of sample geometries for the quantum gravity. 

\subsection*{Acknowledgements} We are grateful to Michael Hinz and an anonymous referee for many helpful comments.

\section{Equilateral barycentric subdivisions and their limit\red{\ (Sasha, Aaron, Ryan et al.)}}\label{sec:bcs}

The process of subdividing the $2$-simplex using its barycentric coordinates is useful in order to establish an isomorphism of graphs in later sections.  Information on barycentric subdivision of more general $n$-simplexes can be found in \cite{Hat02}.  We adapt Hatcher's notation slightly, which is outlined in the following definitions.

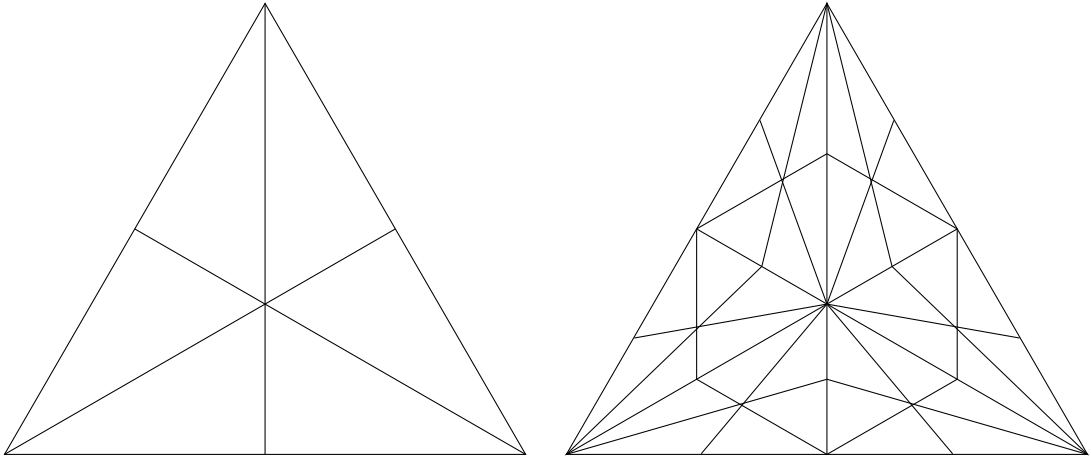
\begin{figure}[t]
\begin{tikzpicture}[scale=1]
\draw (90:4cm)--(210:4cm)--(330:4cm) -- cycle;

\foreach \a in {0,1,2}{
\draw ($(120*\a+90:4cm)$) -- ($(120*\a-90:2cm)$);
}
%
\end{tikzpicture}
\quad
\begin{tikzpicture}[scale=1]
\draw
(90:4cm)--(210:4cm)--(330:4cm) -- cycle;

\foreach \a in {0,1,2}{
\draw
($(120*\a+90:4cm)$) -- ($(120*\a-90:2cm)$);
}

\draw (30:1cm)--(90:4cm)--(150:1cm)--(210:4cm)--(270:1cm)--(330:4cm) -- cycle;
\draw (-30:2cm)--(-90:2cm)--(-150:2cm)--(-210:2cm)--(-270:2cm)--(-330:2cm) -- cycle;

\foreach \a in{0,2,4}{
\draw (0,0)--($(60*\a-10:2.6cm)$);
}

\foreach \a in{1,3,5}{
\draw (0,0)--($(60*\a+10:2.6cm)$);
}
\end{tikzpicture}
\caption{Barycentric subdivision}
\end{figure}

Consider any $2$-simplex (triangle) $T_0$ in the plane, defined by the vertices $[v_0,v_1,v_2]$ which do not all lie on a common line.  The sides of $T_0$ are the $1$-simplexes: $[v_0,v_1], [v_0,v_2], [v_1,v_2]$.

\begin{defn}\label{defn-2-1}
We perform \emph{barycentric subdivision (BCS)} on $T$ as follows:  First, we append the vertex set with the barycenters of the $1$-simplexes $[v_0,v_1], [v_0,v_2], [v_1,v_2]$ and label them $b_{01}, b_{02}, b_{12}$, respectively.  Also append the barycenter of $T_0$ which is the point in the plane given by $\frac{1}{3}(v_0+v_1+v_2)$, which we denote $b$. Thus, $b_{ij}$ is the midpoint of the segment $[v_i, v_j]$.  Any $2$-simplex in the collection of $2$-simplexes formed by the set $N=\set{v_0, v_1, v_2, b_{01}, b_{02}, b_{12}, b}$ is said to be \emph{minimal} if its edges contain no points in $N$ other than its three vertices.  Let $B(T_0)$ denote this collection of minimal $2$-simplexes.  Note that these six triangles are of the form $[v_i, b_{ij}, b]$ where $i\neq j\in\set{1,2,3}.$

We define the process of performing \emph{repeated barycentric subdivision} on $T_0$ as follows:  For a collection $C$ of 2-simplexes, we define $B(C)=\bigcup_{c\in C}B(c)$ to be the collection of minimal $2$-simplexes obtained by performing BCS on each element of $C$.  In this way we define the $n^{th}$ level barycentric subdivision of $T_0$ inductively by $B^n(T_0)=B(B^{n-1}(T_0))$.
\end{defn}

\begin{defn}\label{defn-2-2}
We call the elements of $B^n(T_0)$ the \emph{level n offspring} of $T_0$ where $T_0$ is the \emph{level n ancestor} of its $6^n$ offspring in $B^n(T_0)$.  Similarly, for any triangle $T$ obtained from repeated BCS of $T_0$, we may consider the \emph{level n offspring of T} to be the collection $B^n(T)$.  We use the terms \emph{child}, (resp. \emph{grandchild}) to denote the level $1$ (resp. level $2$) offspring of $T$.  Likewise, we use the terms \emph{parent}, (resp. \emph{grandparent}) to denote the level $1$ (resp. level $2$) ancestor of $T$.  We will use $t\subset T$ to denote that $t$ is a child of $T$, and when necessary $t\subset T \subset T^\prime$ to denote that $t$ is a child of $T$ and a grandchild of $T^\prime$.  If $s$ and $t$ are both children of $T$, then we say that $s$ and $t$ are \emph{siblings}.
\end{defn}

\begin{defn}\label{defn-2-3}
For any triangle $T = [a,b,c]$ we define the \emph{boundary} of $T$ to be the union of its sides, which we denote $\partial T=[a,b]\cup[b,c]\cup[a,c]$.  A level $k$ offspring $t$ of $T$ is said to be a \emph{boundary triangle for T} or \emph{on the boundary of T} if a side of $t$ lies on $\partial T$.  For a given triangle $T$, we say that a level $k$ offspring of $T$ is \emph{special with respect to T} if it is on the boundary of $T$ and contains a vertex of $T$. Note that all 1st level offspring are special with respect to $T$, and that every special offspring has exactly on special offspring.
\end{defn}

\begin{defn}\label{defn-2-4}
We say that two level $n$ triangles are \emph{adjacent} if they share a side.  Given a level $n$ triangle $T=[v_0, v_1, v_2]$, we know the children of $T$ are of the form $[v_i, b_{ij}, b]$ where $i\neq j\in\set{1, 2, 3}$.  We say that two children of $T$ are \emph{vertex adjacent} if their common side is a segment connecting the barycenter of $T$ to one of the original vertices of $T$ ($[v_i,b]$ for some $i\in\set{1, 2, 3}$) and that two children of $T$ are \emph{side adjacent} if their common side is a segment connecting the barycenter of $T$ to one of the barycenters of the sides of $T$ ($[b_{ij},b]$ for some $i\neq j\in\set{1, 2, 3}$).
\end{defn}

Note that each application of BCS on any triangle $T$ produces six new offspring, so we have $|B(C)|=6\cdot |C|$ for any collection of $2$-simplexes $C$.  Similarly, starting with $T_0$ we have inductively that $|B^n(T_0)|=6^n$.  Thus, there are $6^n$ level $n$ offspring of $T_0$.  Also, we see that each triangle in $B^n(T_0)$ is adjacent to at most three other triangles in $B^n(T_0)$ and that if $t\in B^n(T)$ is not on the boundary of $T$, then $t$ is adjacent to exactly three other members of $B^n(T)$.  On the other hand, if $t\in B^n(T)$ is on the boundary of $T$, then $t$ is adjacent to exactly two other members of $B^n(T)$, namely its vertex adjacent sibling and side adjacent sibling, and possibly one other triangle not in $B^n(T)$ but adjacent to $T$.

\begin{prop}\label{prop-2-5}
The following are immediate consequences of the above definitions.
\begin{enumerate}
\item The number of level $n$ boundary triangles of $T_0$ is $6\cdot 2^{n-1}$.
\item The number of adjacencies among $B^n(T_0)$ is $\frac{1}{2}[2\cdot(6\cdot2^{n-1})+3\cdot(6^n-6\cdot2^{n-1})]=2^{n-1}\cdot (3^{n+1}-3)$.
\item If $s\subset S$ and $t\subset T$ are adjacent, then either $S=T$ or $S$ is adjacent to $T$.
\item If $t\subset T$, then exactly one child of $t$ is special with respect to $T$.
\end{enumerate}
\end{prop}

\begin{thm}\label{thm-2-6}
On the triangle $T_0$ there exists a unique geodesic distance $d_\infty(x, y)$ such that each
edge of each triangle in the subdivision $B_n(T_0)$ is a geodesic of length $2^{-n}$. The metric space
$(T_0, d_\infty(x, y))$ has the following properties:
\begin{enumerate}
	\item $(T_0, d_\infty(x, y))$ is a compact metric space homeomorphic to $T_0$ with the usual Euclidean
metric $|x - y|$;
	\item the distance $d_\infty(x, y)$ from any vertex of a triangle in $B_n(T_0)$ to any point on the opposite
side of that triangle is $2^{-n}$;
	\item there are infinitely many geodesics between any two distinct points;
	\item $(T_0, d_\infty(x, y))$ is a self-similar set build with $6$ contracting similitudes with contracting
ratios $\frac12$;
	\item The Hausdorff and self-similarity dimensions of $(T_0, d_\infty(x, y))$ are equal to $\frac{\log(6)}{\log(2)}$.
\end{enumerate}
	
\end{thm}
\begin{proof}
We choose a particular triangle  $T_0 = [v_0,v_1,v_2]$ and construct a metric $d_\infty$ on $T_0$ by approximating with metrics $d_n$ on  $\partial B_n(T_0)$, the union of edges of the triangles in $B_n(T_0)$. We have that $\partial B_n(T_0)$ is a $1$-simplicial
complex, a quantum graph, and a one dimensional manifold with junction points at the vertexes of $B_n(T_0)$. There is a unique
geodesic metric $d_n(x, y)$ on $\partial B_n(T_0)$ such that the length of each edge is $2^{-n}$. 
It is easy to see that for any $x, y \in  \partial B_n(T_0)$ and any $k > 0$ we have 
$d_n(x, y) \geqslant d_{n+k}(x, y)$.
Moreover, by induction one can show that, if $x, y$  vertices of  $B_n(T_0)$ and  $k > 0$, then we have the compatibility
condition
$d_n(x, y) = d_{n+k}(x, y)$.
Therefore on the union
$\cup_{n=0}^\infty \partial B_n(T_0)$ there is a unique geodesic metric $d_\infty(x, y)$ such that each
edge of each triangle in the subdivision $B_n(T_0)$ is a geodesic of length $2^{-n}$.

Since the diameter, in the Euclidean metric, of 2-simplexes of $r$-times repeated barycentric subdivided equilateral triangle is bounded by $(2/3)^r$, $\cup_{n=0}^\infty \partial B_n(T_0)$ is dense in $T_0$ with respect to the Euclidean topology (see Chapter 2, Section 1 of \cite{Hat02}). Thus $d_\infty$ can
be extended to $T_0$ by continuity.

Next we notice that, with the metric $d_\infty$, $T_0$ is self-similar in that if $T = [a_0,a_1,a_2]$ is a 2-simplex of $B_n(T_0)$, then the map of simplexes $f$ sending $v_i \to a_i$, and extending by linearity, is a homeomorphism with $d_\infty(x,y) = 2^{-n}d_\infty(f^{-1}(x),f^{-1}(y))$.

To prove (2), we note that it is enough to show that for, $x$ a corner of $T_0$, and $y$ a vertex of $B_n(T_0)$ on the side of $T_0$ opposite $x$, $d_\infty(x,y) = 1$. This is sufficient because of self-similarity and that such vertexes become dense in the side of $T_0$ which is opposite $x$. Any such $y$ can be connected to a barycenter of a 2-simplex of $B_{n-1}(T_0)$, call this barycenter $x_{n-1}$, by an edge of length $2^{-n}$. In turn, $x_{n-1}$ can be connected to $x_{n-2}$, a barycenter of a 2-simplex of $B_{n-2}(T_0)$ by an edge of length $2^{-n}$. In turn $x_{n-2}$ can be connected to a barycenter of a 2-simplex of $B_{n-3}(T_0)$, called $x_{n-3}$, by a path of length $2^{-n-1}$. Proceeding by induction we get that $x$ can be connected to $y$ by a path of length
\[
2^{-1}+2^{-2}+\cdots + 2^{-n+1} + 2^{-n}+2^{-n} = 1.
\]
This is precisely $d_\infty(x,y)$ because it can be shown that any path from $x$ to $y$ on $\partial B_n(T_0)$ has to pass through at least $2^n-1$ vertexes. This argument implies, in particular, that the diameter of $T_0$ in the   metric $d_\infty$ is equal to $1$, and so the metric is finite even when continued from $\cup_{n=0}^\infty \partial B_n(T_0)$ to $T_0$. 

The same kind of argument also proves  (3). If $y$ is a vertex of $B_n(T_0)$, then it is a vertex of $B_{n+k}(T_0)$ for all $k\geqslant 0$, and the collection of paths described above differ depending on our choice of $n$ and $k$. It is then easy to construct infinitely many geodesics with the same length between arbitrary distinct points of $T_0$. 

To prove (1), we use the 
 insight into the properties of $d_\infty$ described above. In particular, the $d_\infty$ ball of radius $2^{-n}$ centered at $x$, a vertex of $B_n(T_0)$, is contained in the union of all $2$-simplexes of $B_n(T_0)$, which contain $x$ as a vertex. The union of these triangles clearly contains a Euclidean open set containing $x$, and is contained in the Euclidean ball around $x$ of radius $(2/3)^n$. Since $\cup_{n=1}^\infty B_n(T_0)$ is dense in both metrics, this proves that the metrics are equivalent in the sense  that they induce the same topology. This proves (1), and  has the added bonus of proving that $T_0$ is compact with respect to $d_\infty$, and thus $d_\infty$ is complete.

The compactness of $T_0$ with respect to $d_\infty$ proves (4) because $T_0$ is the union of the $6$ contraction maps from $T_0$ onto each 2-simplex of $B_1(T_0)$, and so it is a self-similar structure (see the next section). Each of these maps has contraction ratio $1/2$, and thus using a calculation found in \cite{edg08}, one discovers that the self-similarity and Hausdorff dimension of $T_0$ is $\log(6)/\log(2)$, which proves (5).
\end{proof}
\section{Self-Similar structures and the Stricharz hexacarpet\red{\ (Dan et al.)}}\label{sec:introduction}

First we introduce some notation that we will use.  We denote $\alb = \set{0,1,\ldots, N-1}$, called an \emph{alphabet}, and
\[
\alb^n = \set{x_1x_2\cdots x_n~|~x_i\in\alb}
\]
will be the words of length $n$. Also, we take
\[
\alb^* = \bigcup_{n=0}^\infty \alb^n \quad \text{and,} \quad\shiftspace = \prod_{i=1}^\infty \alb.
\]
Naturally, the set $\alb$ has discrete topology, and $\shiftspace$ is given the product topology (i.e. the topology whose basis is sets of the form $\prod_{i=1}^\infty A_i$, such that $A_i = \alb$ for $i \geq M$ for some $M$).  There is even a natural metric on $\shiftspace$, defined below.

\begin{proposition*}
Fix a number $r\in(0,1)$. For $w = w_1w_2\cdots$ and $v= v_1v_2\cdots$ in $\shiftspace$, we define $\delta_r(w,v) = r^n$ where $n = \min\set{\ell:w_\ell\neq v_\ell}$ with the convention that $\delta_r(w,w)=0$. Then $\delta_r$ is a metric on $\shiftspace$. Additionally, the maps $\leftshift_i(w) = iw$, for $i\in\alb$, is a contraction with Lipschitz constant $r$.
\end{proposition*}

This is proven as Theorem 1.2.2 in \cite{kig01}. The work \cite{kig01} introduces the the theory of self-similar structures and is developed in the context of contractions on metric spaces. We also use the definition of self-similar structure set forth in the above paper. In the rest of this section, we shall take $\delta = \delta_{1/2}$, to be the metric that makes $\shift_i$ contractions with Lipschitz constants $1/2$.

\begin{prop}\label{prop-3-1}
Let $K$ be a compact metrizable space, let $\alb$ be a finite indexing set for $F_i:K\to K$ continuous injections such that $K = \cup_{i\in\alb}F_i(K)$. We call the triplet $\sss = (K,\alb,\set{F_i}_{i\in\alb})$  a \emph{self-similar structure on $K$} if there is a continuous surjection $\pi:\shiftspace\to K$ such that the relation $F_i\circ \pi = \pi\circ \sigma_i$, where $\leftshift_i(w)= iw$, holds for all $w\in\alb^*$.
\end{prop}

We define the $n^{th}$ level \emph{cells} of $K$, $K_w=F_w(K)$ for $w=w_1w_2\cdots w_n\in\alb^n$, where $F_w = F_{w_1}\circ F_{w_2}\circ\cdots\circ F_{w_n}$.  In particular, if $K$ is some quotient space of $\shiftspace$, where $\pi$ is the quotient map where $\sigma_i$'s are constant on the fibres of $\pi$ for all $i\in \alb$, then we can define  $F_i = \pi\circ\shift_i\circ\pi^{-1}$,  creating a self-similar structure. It is in this way we shall define a fractal.

We define the equivalence relation $\sim$ by the following relations: Let $\alb = \set{0,1,\ldots,5}$ where $x$ is any element in $\alb^*$ and $v\in\set{0,5}^\omega$.  Suppose  $i,j\in \alb$ and $j = i+1 \mod 6$.  Then, if $i$ is odd,

\begin{align}\label{eqn:groupequivalence1}
xi3v \sim xj3v \quad\text{ and }\quad xi4v\sim xj4v.
\end{align}
If  $i$ is even ($j$ is still $i+1 \mod 6$), then
\begin{align}\label{eqn:groupequivalence2}
xi1v \sim xj1v \quad \text{ and } \quad xi2v \sim xj2v.
\end{align}
We define $\hex := \shiftspace/\sim$.

We may also define $\hex$ in an alternate way, which we shall call $\tilde{\hex}$. 
The equivalence relation on $\tilde{\hex}$ is defined by $$xiy \sim xjz$$ for $x\in\alb^*$, $i,j\in\alb$, and $z,y\in\shiftspace$, where $$j = i+1 \mod 6$$ and, if  
 $i$ is odd, then  
\begin{align}\label{eqn:fixedpointequivalence1}
y_k = i+2 \text{ or } i+3 \mod 6\quad \text{ and }\quad z_k = \begin{cases}
i-1 \mod 6 & \text{if } y_k = i+2\mod 6 \\
i-2 \mod 6 & \text{if } y_k = i+3 \mod 6.
\end{cases}
\end{align}
and if $i$ is even, then 
\begin{align}\label{eqn:fixedpointequivalence2}
y_k = i+1 \text{ or } i+2 \mod 6\quad \text{and }\quad z_k = \begin{cases}
i  & \text{if } y_k = i+1\mod 6 \\
i-1 \mod 6 & \text{if } y_k = i+2 \mod 6.
\end{cases}
\end{align}

\begin{prop}
The space $(\hex, \alb, \set{\shift_i})$ is a self-similar structure, where $\shift_i:\hex\to\hex$ is defined, if $\pi: \shiftspace\to\hex$ is the projection associated with the equivalence relation $\sim$, then $\shift_i(\pi(x)) = \pi(\shift_i(x))$. 
\end{prop}

\begin{proof}
Since $x\sim y$ implies that $\shift_i(x)\sim\shift_i(y)$ for all $x,y\in\shiftspace$ and $i\in\alb$ (this can be seen in the definition of $\sim$), $\shift_i:\hex \to \hex$ is well defined. The only thing left to check is that $\hex$ is metrizable. This follows because we can define a metric $\delta$ on $\hex$ by
\[
\delta([x],[y]) = \inf_{x\in[x],y\in[y]}\set{\delta_r(x,y)}.
\]
This metric is well defined because the equivalence classes of $\sim$ contain at most 2 elements.
\end{proof}

\begin{figure}[t]
\begin{tikzpicture}[scale=1.]
\foreach \a in {0,...,5}{
\node[regular polygon, regular polygon sides=6, minimum size=2.cm, draw] at ($(\a*60+30:2.8cm)$) {~};

\draw ($(\a*60+30:2.8cm)+(\a*60+180:1cm)$) --
      ($(\a*60+90:2.8cm)+(\a*60+300:1cm)$)
      ($(\a*60+30:2.8cm)+(\a*60+120:1cm)$) --
      ($(\a*60+90:2.8cm)+(\a*60:1cm)$)      ;

\foreach \b in {0,...,5}{
\draw[fill=white,thick] ($(\a*60+30:2.8cm)+(\b*60+\a*60:1cm)$) circle (2pt);
}
}

\foreach \a in {1,3,5}{
\foreach \b in {0,...,5}{
\node[regular polygon, regular polygon sides=6, minimum size=1.5cm] at ($(\a*60+30:2.8cm)+(-\b*60+\a*60:.7cm)$) {$\a\b$};
}}

\foreach \a in {0,...,5}{

\node at ($(\a*60+30:2.8cm)$) {$\a$};
}

\foreach \a in {0,2,4}{
\foreach \b in {0,...,5}{
\node[regular polygon, regular polygon sides=6, minimum size=1.5cm] at ($(\a*60+30:2.8cm)+(\b*60+\a*60+60:.7cm)$) {$\a\b$};
}}
\end{tikzpicture}
~
\begin{tikzpicture}[scale=1.]
\foreach \a in {0,...,5}{
\node[regular polygon, regular polygon sides=6, minimum size=2.cm, draw] at ($(\a*60+30:2.8cm)$) {~};

\draw ($(\a*60+30:2.8cm)+(\a*60+180:1cm)$) --
      ($(\a*60+90:2.8cm)+(\a*60+300:1cm)$)
      ($(\a*60+30:2.8cm)+(\a*60+120:1cm)$) --
      ($(\a*60+90:2.8cm)+(\a*60:1cm)$)      ;

\foreach \b in {0,...,5}{
\draw[fill=white,thick] ($(\a*60+30:2.8cm)+(\b*60+\a*60:1cm)$) circle (2pt);
}
}

\foreach \a in {0,2,4}{
\foreach \b in {0,...,5}{
\node[regular polygon, regular polygon sides=6, minimum size=1.5cm] at ($(\a*60+30:2.8cm)+(\b*60:.7cm)$) {$\a\b$};
}}

\foreach \a in {1,3,5}{
\foreach \b in {0,...,5}{
\node[regular polygon, regular polygon sides=6, minimum size=1.5cm] at ($(\a*60+30:2.8cm)-(\b*60-120:.7cm)$) {$\a\b$};
}}

\foreach \a in {0,...,5}{

\node at ($(\a*60+30:2.8cm)$) {$\a$};
}


\end{tikzpicture}
\caption{$\face_2$ generated by \eqref{eqn:groupequivalence1}, \eqref{eqn:groupequivalence2} on left, and \eqref{eqn:fixedpointequivalence1}, \eqref{eqn:fixedpointequivalence2} on right.}
\end{figure}

\begin{prop}
The equivalences defined in equations \eqref{eqn:groupequivalence1},\eqref{eqn:groupequivalence2} and equation \eqref{eqn:fixedpointequivalence1}, \eqref{eqn:fixedpointequivalence2} provide two definitions of $\hex$ which are equivalent. 
\end{prop}

\begin{proof}
We show that there is a self homeomorphism $f:\shiftspace\to\shiftspace$ which transforms the equivalence in \ref{eqn:groupequivalence1} and \ref{eqn:groupequivalence2} into the equivalence \ref{eqn:fixedpointequivalence1} and \ref{eqn:fixedpointequivalence2}. This map $f$ is given by $f(u) = v$, where
\[
v_1 := u_1\quad\text{and}\quad v_m := (-1)^{\alpha_{m-1}}u_m + \alpha_{m-1},
\]
where $\alpha_k = \sum_{j=1}^k u_j$. This map is continuous,  since how $f$ acts on the $n$th letter is independent of any future letters, $\delta(x,y) \leq (f(x),f(y))$. In fact, since $f$ acts bijectively on $\alb^n$ (which is easy to check by induction), this means that $f$ is an isometry with respect to $\delta_r$, and is thus a bijection as a map of from $\shiftspace$ to itself. Moreover, $f$  also descends to an isometry with respect to $\delta$, but to see this we need to know that $f$ preserves the equivalence relation, which we will show shortly.

Another way of seeing that $f$ is bijective comes from that fact that we can define an inverse $f^{-1}(v) = u$ where
\[
u_1:=v_1\quad\text{and}\quad u_m := (-1)^{\alpha_{m-1}}(v_m - \alpha_{m-1}).
\]
Showing that the two equivalences produce identical quotient spaces is a matter of showing that $f$ and $f^{-1}$ preserve equivalence, i.e. $u\sim v$ with respect to \ref{eqn:groupequivalence1}, \ref{eqn:groupequivalence2} implies that $f(u)\sim f(v)$ with respect to \ref{eqn:fixedpointequivalence1},\ref{eqn:fixedpointequivalence2}. This involves checking a few cases, we provide an example of such a case.

For the remainder of the proof, all equalities are assumed to be mod $6$. Now suppose that $u^{(1)}$ is of the form $xi3w$ and $u^{(2)}$ is of the form $xj3w$  where the $j = i+1$ and $i$ is odd, as in \ref{eqn:groupequivalence1}, and we assume that the $i$ and $j$ are in the $n$th position. Finally we assume, letting $\alpha^{(1)}_k = \sum_{l=1}^k u^{(l)}_l$ and $\alpha^{(2)}_k = \sum_{l=1}^k u^{(2)}_l$, that $\alpha^{(1)}_{n-1} = \alpha^{(2)}_{n-1}$ is odd. We need to vary all of the above for various cases of equivalence.

Let $v^{(l)} = f(u^{(l)})$ for $l=1,2$, then for $k< n$, $v_k^{(1)}=v_k^{(2)}$, and $v^{(1)}_n = v^{(2)}_n-1$.  Furthermore, $v_n^{(2)}$ and $\alpha_n^{(2)}$ is odd, so $v^{(2)}_{n+1} = v^{(2)}_{n} + 3 $, and
 \[v_{n+1}^{(1)} =v_{n}^{(1)}+3= v^{(2)}_{n-1} - (j-1) + 3 = v^{(2)}_{n} +4 = v^{(2)}_{n} -2 \mod 6, 
\] 
this is consistent with \ref{eqn:fixedpointequivalence1}.
 
 For $k> n+1$, if we have $u^{(1)}_k=u^{(2)}_k = 0$, then $v^{(1)}_k=u^{(1)}_{k-1}$ and $v^{(2)}_k=v^{(2)}_{k-1}$. The first instance (if it happens at all) where $u^{(1)}_k=u^{(2)}_k = 5 $, $\alpha^{(2)}_{k-1} = \alpha^{(2)}_n + 3$ will be odd, so $ v^{(2)}_k = v^{(2)}_n + 2$ and $\alpha^{(1)}_{k-1} = \alpha^{(1)}_n + 3$ will be even, so $v^{(1)}_k = v^{(2)}_n-1$. Any further instance where $u^{(1)}_k=u^{(2)}_k = 5$, with add or subtract 1 to alternatively and any instance where $u^{(1)}_k=u^{(2)}_k = 0 $ will add nothing. In this way, we have $v^{(1)}_k=v^{(2)}_n + 2$ or  $v^{(2)}_n + 3$ with $v^{(1)}_k=v^{(2)}_n -1$ or  $
 v^{(2)}_n - 2 $ respectively. This shows that $v^{(1)} \sim v^{(2)}$ in the sense of 
 \ref{eqn:fixedpointequivalence1}.

It is in this way that all cases can be checked.
\end{proof}

We define the \emph{shift map} $\shift:\shiftspace\to\shiftspace$ by $\shift(x_1x_2x_3\cdots) = x_2x_3\cdots$. In this way, the concatenation maps $\shift_i$ can be seen as branches of the inverse of $\shift$.

\begin{lem}\label{lem-3-2}
If we use the definition of $\hex$ from equation \ref{eqn:groupequivalence1} and \ref{eqn:groupequivalence2}, the shift map $\shift$ descends to a well defined map from $\hex$ to $\hex$, which we also call $\shift$. 
\end{lem}

The proof of the above is a matter of showing, for $x,y\in\shiftspace$ that  $x\sim y$ if and only if $\shift(x)\sim\shift(y)$. This is easily verified to be true in the equivalences in \ref{eqn:groupequivalence1} and \ref{eqn:groupequivalence2}.

We turn to topological properties of the space, a natural question is what does the intersection of two neighboring cells look like.

\begin{prop}\label{cantor ramification}
If $i,j\in\alb$ and $i\neq j$, then either $\sigma_i(\hex) \cap \sigma_j(\hex)$ is empty or homeomorphic to the middle third Cantor set.
\end{prop}

\begin{proof}
The set $\sigma_i(\hex)$ consists of infinite words beginning with the letter $i$, the intersection with $\sigma_j(\hex)$ is the set of words which begin with an $i$ which are equivalent to the words which start with a $j$. If $i \neq j\pm 1$, then this intersection is empty, Since there is no loss in generality, we assume that $j=i+1$, if we further assume that $i$ is odd, then \ref{eqn:groupequivalence1} tells us that $\sigma_i(\hex)\cap\sigma_j(\hex)$ is given by the words
\[
i3v\sim j3v \quad\text{or}\quad i4v \sim j4v
\]
where $v\in\shiftspace$ is an infinite word consisting of $0$'s and $5$'s. By ignoring the leading $i$ or $j$, this set is naturally homeomorphic to the shift space of $\set{0,1}$ --- which is in turn homeomorphic to the middle third Cantor set (see, for example, \cite{kig01}). The case where $i$ is even follows the exact same argument.
\end{proof}

If we examine planar realizations of the approximating graphs, we see a large ``hole'' in the center. We see that this hole consists of truncation of the elements of the set $C$ consisting of words of the form $ijv$ where $i,j\in\alb$ where $i$ is any letter and $j$ is a $3$ or $2$, and $v\in\shiftspace$ is a word consisting of $0$'s and $1$'s.

\begin{prop}
The set $C$ is homeomorphic to the circle.
\end{prop}

\begin{proof}
If we consider $C\cap\shift_i(\hex)$ (the subset starting in $i$), then we have that $i21\overline{0}\sim i31\overline{0}$ and $ix301\overline{0}\sim ix311\overline{0}$. for any finite word $x$. The shift space of $\set{0,1}$ with $x01\overline{0} \sim x11\overline{0}$ is homeomorphic to the unit interval (this is seen in the limit space of the Grigorchuk group, see \cite{Nek05} section 3.5.3).

In this homeomorphism, the endpoints of the interval are $\overline{0}$ and $1\overline{0}$. So we have two copies of the unit interval corresponding to the sets starting with $i2$ and $i3$, which are identified at one of the endpoints. This shows that each set $C\cap\shift_i(\hex)$ is itself isomorphic to the interval. These intervals are in turn identified at their endpoints, $i2\overline{0}\sim k2\overline{0}$ and $i3\overline{0}\sim \ell3\overline{0}$, where $k = i\pm 1$ and $\ell = i\mp 1$, depending on whether $i$ is odd or even.
\end{proof}

From the results above we obtain, in particular, the following theorem. 

\begin{thm}\label{thm-K}
The self-similar set $K$, defined above and called the Stichartz hexacarpet,  is an infinitely ramified fractal not homeomorphic to $T_0$.
\end{thm}

\begin{proof}
This follows by proposition \ref{cantor ramification}. Since $K$ can be disconnected into arbitrarily small pieces by removing a topological Cantor set, it must be topologically one dimensional. On the other hand, $T_0$ is topologically 2 dimensional, thus the spaces cannot be homeomorphic. The general theory of topological dimension can be found in \cite{engelking} (in particular see Definitions 
1.1.1 and 1.6.7, Theorems 1.4.5, 4.1.4 and 4.1.5). 
\end{proof}

\section{Graph Approximations and Isomorphism\red{\ (Aaron and Ryan et al.)}}\label{sec:graphiso}

In this section we show that the self-similar structure from the last section can be approximated by graphs constructed from repeated barycentric subdivision of a 2-simplex.

We make this precise by constructing approximating graphs to the hexacarpet $\hex$. Taking $\alb=\set{0,1,\ldots,5}$ as in the last section, we define the graph $\face_n=(\alb^n,\edge_n)$ with $\alb^n$ as the set of vertices. We define the edge relations $\edge_n$ where two words $u$ and $w$ are connected if there are $x,y\in\shiftspace$ such that the concatenated words $wx\sim uy$ according to the equivalence defining $\hex$, as in equations \eqref{eqn:groupequivalence1} and \eqref{eqn:groupequivalence2}. We can alternatively define the the set of vertices to be the set of $n^{th}$ level cells of $\hex$, where $(\hex_w,\hex_u)$ are in the edge relation if $\hex_w\cap\hex_u \neq \emptyset$. In this way, we can think of the vertices of $\face_n$ as being $n$th level cells of $\hex$.

We now exhibit partitions of the vertex set $\alb^n$ and edge set $\edge_n$ which will be useful in discussing edge relations:  Let $\mathsf{W}_{1}= \set{x_1x_2\cdots x_n~|~x_i\in\set{0,5}, 2 \leq i \leq n}$ be the set of words of length $n$ whose second through last letters are 0's or 5's.  For each $x=x_1x_2\cdots x_n\in\alb^n\backslash \mathsf{W}_{1}$ there is at least one $i\in\set{2,3,\cdots ,n}$ such that $x_{i}\notin\set{0,5}$.  So we may define the function $l:\alb^n\backslash \mathsf{W}_{1}\to\set{2,3,\cdots ,n}$ by $l:x_1x_2\cdots x_n\to\max\set{2 \leq i \leq n~|~x_{i}\notin\set{0,5}}$.  Now for $2 \leq k \leq n$ define $\mathsf{W}_{k}=l^{-1}\left(k\right)=\set{x_1x_2\cdots x_n~|~x_k \notin\set{0,5}, x_{i} \in\set{0,5}, k<i\leq n}$.  These are the words which end in exactly $(n-k)$ $0$s and $5$s.

From the equivalence relation $\sim$ defined on $\shiftspace$ we recover the following edge relations on $\alb^n$ by truncating the relations after the $n^{th}$ coordinate:

\begin{itemize}
\item[] $\mathsf{F}_1=\set{\{xi,xj\}~|~x\in\alb^{n-1}, j=i+1\mod6}$.
\item[] $\mathsf{F}_k=\set{\{xi\alpha v,xj\alpha v\}~|~x\in\alb^{k-2}, j=i+1\mod6, \alpha = 3 ~\text{or}~4,~v\in\set{0,5}^{n-k}}$, for $2\leq k\leq n$.
\end{itemize}

We simplify the edge relations by writing $\set{xi\alpha v,xj\alpha v}$.  Here $\alpha\in\set{3,4}$ is the same in both components of the relation when $i$ odd, $j=i+1\mod6$.  We take $\alpha\in\set{1,2}$ to be the same on both sides when $i$ even, $j=i+1\mod6$.

We now collect some information about the cardinalities of the vertex and edge sets.

\begin{prop}\label{prop-4-1}
The following are apparent from our construction:
\begin{enumerate}
\item$|\mathsf{W}_{1}|=3\cdot2^n$ and $|\mathsf{W}_{k}|=6^{k-1}\cdot4\cdot2^{n-k}=3^{k-1}\cdot2^{n+1}$, for $2\leq k\leq n$.
\item$|\mathsf{F}_1|=6^n$ and $|\mathsf{F}_k|=6^{k-1}\cdot2^{n-k+1}$, for $2\leq k\leq n$.
\item$|\edge_n|=\sum_{k=1}^{n}|\mathsf{F}_k|=2^{n-1}\cdot (3^{n+1}-3)$.
\end{enumerate}
\end{prop}

\begin{prop}\label{prop-4-2}
Every vertex in $\mathsf{W}_1$ has degree two, with both edge relations in $\mathsf{F}_1$.
For $k\geq 2$, every vertex in $\mathsf{W}_k$ has degree three, with two edge relations in $\mathsf{F}_1$ and the third in $\mathsf{F}_k$.
\end{prop}

\begin{proof}
First we note that each vertex $x=x_1x_2\cdots x_n$ has exactly two edge relations in $\mathsf{F}_1$, namely $\{x_1x_2\cdots x_{n-1}y,x\}$ and $\{x,x_1x_2\cdots x_{n-1}z\}$ where $y=x_n-1 \mod6$ and $z=x_n+1 \mod6$.  In addition, for all $2\leq k\leq n$ each vertex in $\mathsf{W}_k$ has one additional edge relation in $\mathsf{F}_k$.   By construction, the vertices in $\mathsf{W}_1$ do not have any edge relations in $\mathsf{F}_k$ for all $k\neq 1$.
\end{proof}

We now collect our information on the graph approximation of the hexacarpet and the graph constructed using repeated Barycentric subdivision of the $2$-simplex in order to establish an isomorphism between them. We begin by introducing some information which will be useful in the proof of Theorem~\ref{thm-4-7} at the end of this section.

\begin{prop}\label{prop-4-3}
There exists a labeling of $B^n(T_0)$ with the strings in $\alb^n$ that establishes a bijection between the two sets.
\end{prop}

\begin{proof}
Label $B(T_0)$ with the elements in the alphabet $\alb=\set{0,1,2,3,4,5}$ cyclically so that we have the edge adjacencies $\set{1,2}, \set{3,4}, \set{5,0}$ and the vertex adjacencies $\set{0,1}, \set{2,3}, \set{4,5}$.  We call this construction a \emph{standard labeling} of the offspring of $T_0$ with the letters of the alphabet $\alb$.  We note that there are six standard labelings of the offspring of $T_0$, and each is uniquely determined by labeling any one child.

  We choose an arbitrary child of $T_0$ to be labeled 0 and construct a standard labeling of the remaining children.  Thus we have labeled the triangles of $B(T_0)$ bijectively with the words in $\alb^1$ via the standard labeling map $\Phi_1:B(T_0)\to\alb$.  For $n\geq 2$, we shall define an inductive labeling of $B^n(T_0)$ with the words in $\alb^n$ and establish the bijection $\Phi_n:B^n(T_0)\to\alb^n$.

By assumption, for each triangle $t\in B^{n-1}(T_0)$ we have an associated unique word, $x=x_1x_2 \cdots x_{n-1}$ in $\alb^{n-1}$ (ie. $\Phi_{n-1}:B^{n-1}(T_0)\to\alb^{n-1}$ is a bijection).  
We will label the offspring of $t$ with the words $x0,x1,x2,x3,x4,x5$ as follows (where $x0$ denotes the word of length $n, x_1x_2\cdots x_{n-1}0\in\alb^n$):

Let $T$ be the parent of $t$.  From above we know that exactly one child of $t$ is special with respect to $T$.  We assign this triangle the word $x0$ and label the other children of $t$ according to the standard labeling fixed by $x0$.  Therefore, to each element of $B^n(T_0)$ we have associated a word of $\alb^n$.

To show that this is an injective labeling, assume that there are two level $n$ offspring $s$ and $t$ of $T_0$ which have the same label, say $x_1x_2\cdots x_n\in\alb^n$.  By the induction assumption, there is exactly one triangle $T\in B^{n-1}(T_0)$ with the labeling $x_1x_2\cdots x_{n-1}$.  This means that $s$ and $t$ are both children of $T$ which have the same label $x_n$ by the standard labeling of the children of $T$.  Thus, $s$ and $t$ are the same triangle.  So we have $\Phi_n$ is an injective map between the finite sets $B^n(T_0)$ and $\alb^n$.  Since $|B^n(T_0)|=|\alb^n|=6^n$, we see that $\Phi_n$ is a bijection as desired.
\end{proof}

\begin{defn}\label{defn-4-4}
For any triangle $s\in B^n(T_0)$ there exists a unique chain of ancestors, $s=s_n\subset s_{n-1}\subset\cdots\subset s_1\subset s_0=T_0$ called \emph{the family tree of s}.
\end{defn}

\begin{prop}\label{prop-4-5}
If $s,t\in B^n(T)$ and $s$ is adjacent to $t$ then there is some maximal $0\leq m\leq n$ such that $s_m=t_m$ in the family trees for $s$ and $t$ (i.e. $s$ and $t$ are both level $\ell$ offspring of $s_m=t_m$.)  In particular, as k level triangles, $s_k$ and $t_k$ are equal for all $0\leq k\leq m$ and are adjacent for all $m< k\leq n$.
\end{prop}

\begin{proof}
For any adjacent $s$ and $t$ we have that $s_0=t_0=T_0$ so this assignment is well defined.  Since each triangle has a unique parent, if $s_m=t_m$ for some $0\leq m\leq n$, then $s_{m-1}=t_{m-1}$.  Thus, we have $s_k$ and $t_k$ are equal for all $0\leq k\leq m$.  To see that $s_k$ and $t_k$ are adjacent for all $m< k\leq n$ we note that if $s_k\subset s_{k-1}$ and $t_k\subset t_{k-1}$ are adjacent, then $s_{k-1}$ and $t_{k-1}$ are either adjacent or equal.  For all $n\geq k\geq m+2$ we have $s_{k-1}$ and $t_{k-1}$ are not equal by assumption.  Also by assumption, $s_n$ and $t_n$ are adjacent.  Therefore, by induction we see that the final statement of the proposition holds.
\end{proof}

\begin{lem}\label{lem-4-6}
Let $s$ and $t$ be adjacent level $n$ offspring of $T_0$ with $s\subset S\subset S^\prime$ and $t\subset T\subset T^\prime$.  Then exactly one of the following is true of the labeling of $s$ and $t$.

\begin{enumerate}
\item If $S=T$ has the label $x\in\alb^{n-1}$, then, either $s$ has label $xi$ and $t$ has label $xj$ for some $i\in\alb, j=i+1\mod6$, or 
$t$ has label $xi$ and $s$ has label $xj$.  So the addresses of $s$ and $t$ differ only in their last letter by $1\mod6$.

\item If $S\neq T$ and $S^\prime=T^\prime$ has the labeling $x\in\alb^{n-2}$, then $s$ is labelled $xi\alpha$ and $t$ is labelled  $xj\alpha$.  If $i$ is even, then $\alpha$ is either a $1$ or $2$ (in the addresses of both $s$ and $t$) and if $i$ is odd, then $\alpha$ is either a $3$ or $4$ (in both addresses).

\item If $S\neq T$ and $S^\prime\neq T^\prime$, then the addresses for $s$ and $t$ end in the same letter, which is either a $0$ or $5$.
\end{enumerate}
\end{lem}

\begin{proof}
The proof of (1) is immediate from the standard labeling procedure.  We have forced the children of every triangle to observe the adjacencies $\set{i,j}$ for all $j=i+1\mod6$.

For (2), we see that $S$ and $T$ satisfy the hypothesis of (1), therefore without loss of generality we may assume $S=xi$ and $T=xj$ for some $j=i+1\mod6$.  Now $i$ is either even or odd.  If $i$ is even, then we have $S$ and $T$ are vertex adjacent siblings (this was also forced by our construction).  Thus, their common edge is of the form $u=[v_k,b]$, $k \in\set{0, 1, 2}$, where $v_k$ (resp. $b$) is a vertex (resp. the barycenter) of $S^\prime=T^\prime$.  By the standard labeling of the offspring of $S$ and $T$, we see that the labels of $s$ and $t$ end in either a $1$ or $2$.  Assume for contradiction that the last coordinate of $s$ and $t$ are different.  Thus, without loss of generality, we have that $s=xi1$ and $t=xj2$.  We see that $s=[v_k, b_{S}, b_u]$ and $t=[b, b_{T}, b_u]$  where $b_u$ is the barycenter of the segment $u$ and $b_{S}$ (resp. $b_{T}$) is the barycenter of $S$ (resp. $T$).  Since $s$ and $t$ have only one vertex in common, it is impossible for $s$ and $t$ to be adjacent.  Now we see that the address for $s$ is $xi\alpha$ and the address for $t$ is $xj\alpha$ where $\alpha\in\set{1,2}$ is the same in both addresses. If $i$ is odd, then $S$ and $T$ are side adjacent siblings and the proof follows as in the case when $i$ is even.

For (3), we know that $S^\prime$ and $T^\prime$ are adjacent, so without loss of generality, let their common side be $[v_0,v_1]$.  After BCS we see that $S$ and $T$ must have either $[v_0,b_{01}]$ or $[v_1,b_{01}]$ as their common side.  Again without loss of generality, assume that the common side of $S$ and $T$ is $[v_0, b_{01}]$.  Subdividing further we see that the common side of $s$ and $t$ must be either $[v_0,m]$ or $[m,b_{01}]$, where $m$ is the barycenter of $[v_0, b_{01}]$.  Since $[v_0,m]\subset [v_0,v_1]$, if $[v_0,m]$ is a side of $s$ and $t$ we see that $s$ and $t$ are each special with respect to their respective grandparent, $S^\prime$ and $T^\prime$.  Thus, the last letter in the addresses of both $s$ and $t$ must be $0$ by construction.  On the other hand, if $[m,b_{01}]$ is the common side of $s$ and $t$, we see that $s$ and $t$ are both side adjacent to their siblings who received a label of $0$.  Thus, the last letter in the addresses of $s$ and $t$ must be $5$ by the standard labeling construction.
\end{proof}

We are now in a position to define the desired isomorphism of graphs.

\begin{figure}[t] 

%
\begin{tikzpicture}[scale=1]
\draw
(90:4cm)--(210:4cm)--(330:4cm) -- cycle;

\foreach \a in {0,1,2}{
\draw
($(120*\a+90:4cm)$) -- ($(120*\a-90:2cm)$);
}

\draw (30:1cm)--(90:4cm)--(150:1cm)--(210:4cm)--(270:1cm)--(330:4cm) -- cycle;
\draw (-30:2cm)--(-90:2cm)--(-150:2cm)--(-210:2cm)--(-270:2cm)--(-330:2cm) -- cycle;

\foreach \a in{0,2,4}{
\draw (0,0)--($(60*\a-10:2.6cm)$);
}

\foreach \a in{1,3,5}{
\draw (0,0)--($(60*\a+10:2.6cm)$);
}

\foreach \a in{0,2,4}{

\node at ($.111*(150+120*\a:2cm)+.111*(90+120*\a:4cm) + .333*(90 + 120*\a:2cm) $) {$\a 2 $};
\node at ($.111*(150+120*\a:2cm)+.111*(90+120*\a:4cm) + .333*(90 + 120*\a:2cm)+.333*(90 + 120*\a:4cm) $) {$\a 1 $};
\node at ($.111*(150+120*\a:2cm)+.111*(90+120*\a:4cm) + .333*(90 + 120*\a:4cm) +.333*(120*\a+110:2.6cm)$) {$\a 0 $};

\node at ($.111*(150+120*\a:2cm)+.111*(90+120*\a:4cm) + .333*(150 + 120*\a:1cm) $) {$\a 3 $};
\node at ($.111*(150+120*\a:2cm)+.111*(90+120*\a:4cm) + .333*(150 + 120*\a:2cm) + .333*(150 + 120*\a:1cm) $) {$\a 4 $};
\node at ($.111*(150+120*\a:2cm)+.111*(90+120*\a:4cm) + .333*(150 + 120*\a:2cm) +.333*(120*\a+110:2.6cm)$) {$\a 5 $};
}

\foreach \a in{1,3,5}{

\node at ($.111*(-30+120*\a-60:2cm)+.111*(30+120*\a-60:4cm) + .333*(30 + 120*\a-60:2cm) $) {$\a 2$};
\node at ($.111*(-30+120*\a-60:2cm)+.111*(30+120*\a-60:4cm) + .333*(30 + 120*\a-60:2cm)+.333*(30 + 120*\a-60:4cm) $) {$\a 1 $};
\node at ($.111*(-30+120*\a-60:2cm)+.111*(30+120*\a-60:4cm) + .333*(30 + 120*\a-60:4cm) +.333*(120*\a-60-10:2.4cm)$) {$\a 0 $};

\node at ($.111*(-30+120*\a-60:2cm)+.111*(30+120*\a-60:4cm) + .333*(-30 + 120*\a-60:1cm) $) {$\a 3 $};
\node at ($.111*(-30+120*\a-60:2cm)+.111*(30+120*\a-60:4cm) + .333*(-30 + 120*\a-60:2cm) + .333*(-30 + 120*\a-60:1cm) $) {$\a 4 $};
\node at ($.111*(-30+120*\a-60:2cm)+.111*(30+120*\a-60:4cm) + .333*(-30 + 120*\a-60:2cm) +.333*(120*\a-60-10:2.4cm)$) {$\a 5 $};
}
\end{tikzpicture}
\quad
\begin{tikzpicture}[scale=1]
\draw
(90:4cm)--(210:4cm)--(330:4cm) -- cycle;

\foreach \a in {0,1,2}{
\draw
($(120*\a+90:4cm)$) -- ($(120*\a-90:2cm)$);
}

\draw (30:1cm)--(90:4cm)--(150:1cm)--(210:4cm)--(270:1cm)--(330:4cm) -- cycle;
\draw (-30:2cm)--(-90:2cm)--(-150:2cm)--(-210:2cm)--(-270:2cm)--(-330:2cm) -- cycle;

\foreach \a in{0,2,4}{
\draw (0,0)--($(60*\a-10:2.6cm)$);
}

\foreach \a in{1,3,5}{
\draw (0,0)--($(60*\a+10:2.6cm)$);
}

\foreach \a in{0,2,4}{

\draw[fill=black] ($.111*(150+60*\a:2cm)+.111*(90+60*\a:4cm) + .333*(90 + 60*\a:2cm) $) circle (2pt);
\draw[fill=black] ($.111*(150+60*\a:2cm)+.111*(90+60*\a:4cm) + .333*(90 + 60*\a:2cm)+.333*(90 + 60*\a:4cm) $) circle (2pt);
\draw[fill=black] ($.111*(150+60*\a:2cm)+.111*(90+60*\a:4cm) + .333*(90 + 60*\a:4cm) +.333*(60*\a+110:2.6cm)$) circle (2pt);

\draw[fill=black] ($.111*(150+60*\a:2cm)+.111*(90+60*\a:4cm) + .333*(150 + 60*\a:1cm) $) circle (2pt);
\draw[fill=black] ($.111*(150+60*\a:2cm)+.111*(90+60*\a:4cm) + .333*(150 + 60*\a:2cm) + .333*(150 + 60*\a:1cm) $) circle (2pt);
\draw[fill=black] ($.111*(150+60*\a:2cm)+.111*(90+60*\a:4cm) + .333*(150 + 60*\a:2cm) +.333*(60*\a+110:2.6cm)$) circle (2pt);
}

\foreach \a in{0,2,4}{
\draw[thick]
($.111*(150+60*\a:2cm)+.111*(90+60*\a:4cm) + .333*(90 + 60*\a:2cm) $)--
 ($.111*(150+60*\a:2cm)+.111*(90+60*\a:4cm) + .333*(90 + 60*\a:2cm)+.333*(90 + 60*\a:4cm) $) --
 ($.111*(150+60*\a:2cm)+.111*(90+60*\a:4cm) + .333*(90 + 60*\a:4cm) +.333*(60*\a+110:2.6cm)$)  --
($.111*(150+60*\a:2cm)+.111*(90+60*\a:4cm) + .333*(150 + 60*\a:2cm) +.333*(60*\a+110:2.6cm)$) --
 ($.111*(150+60*\a:2cm)+.111*(90+60*\a:4cm) + .333*(150 + 60*\a:2cm) + .333*(150 + 60*\a:1cm) $) -- 
 ($.111*(150+60*\a:2cm)+.111*(90+60*\a:4cm) + .333*(150 + 60*\a:1cm) $) -- cycle;

}

\foreach \a in{1,3,5}{
\draw[fill=black] ($.111*(-30+60*\a:2cm)+.111*(30+60*\a:4cm) + .333*(30 + 60*\a:2cm) $) circle (2pt);
\draw[fill=black] ($.111*(-30+60*\a:2cm)+.111*(30+60*\a:4cm) + .333*(30 + 60*\a:2cm)+.333*(30 + 60*\a:4cm) $) circle (2pt);
\draw[fill=black] ($.111*(-30+60*\a:2cm)+.111*(30+60*\a:4cm) + .333*(30 + 60*\a:4cm) +.3*(60*\a-10:2.6cm)$) circle (2pt);

\draw[fill=black] ($.111*(-30+60*\a:2cm)+.111*(30+60*\a:4cm) + .333*(-30 + 60*\a:1cm) $) circle (2pt);
\draw[fill=black] ($.111*(-30+60*\a:2cm)+.111*(30+60*\a:4cm) + .333*(-30 + 60*\a:2cm) + .333*(-30 + 60*\a:1cm) $) circle (2pt);
\draw[fill=black] ($.111*(-30+60*\a:2cm)+.111*(30+60*\a:4cm) + .333*(-30 + 60*\a:2cm) +.3*(60*\a-10:2.6cm)$) circle (2pt);
}

\foreach \a in{1,3,5}{
\draw[thick] ($.111*(-30+60*\a:2cm)+.111*(30+60*\a:4cm) + .333*(30 + 60*\a:2cm) $)--
 ($.111*(-30+60*\a:2cm)+.111*(30+60*\a:4cm) + .333*(30 + 60*\a:2cm)+.333*(30 + 60*\a:4cm) $)--
($.111*(-30+60*\a:2cm)+.111*(30+60*\a:4cm) + .333*(30 + 60*\a:4cm) +.333*(60*\a-10:2.4cm)$) --
($.111*(-30+60*\a:2cm)+.111*(30+60*\a:4cm) + .333*(-30 + 60*\a:2cm) +.333*(60*\a-10:2.4cm)$)-- 
 ($.111*(-30+60*\a:2cm)+.111*(30+60*\a:4cm) + .333*(-30 + 60*\a:2cm) + .333*(-30 + 60*\a:1cm) $)--
 ($.111*(-30+60*\a:2cm)+.111*(30+60*\a:4cm) + .333*(-30 + 60*\a:1cm) $)  -- cycle;
}

\foreach \a in {0,2,4}{
\draw[thick] ($.111*(150+60*\a:2cm)+.111*(90+60*\a:4cm) + .333*(90 + 60*\a:2cm) $) --
($.111*(-30+60*\a+60:2cm)+.111*(30+60*\a+60:4cm) + .333*(30 + 60*\a+60:2cm) $);
\draw[thick] ($.111*(150+60*\a:2cm)+.111*(90+60*\a:4cm) + .333*(90 + 60*\a:2cm)+.333*(90 + 60*\a:4cm) $)--($.111*(-30+60*\a+60:2cm)+.111*(30+60*\a+60:4cm) + .333*(30 + 60*\a+60:2cm)+.333*(30 + 60*\a+60:4cm) $) ;

\draw[thick] ($.111*(150+60*\a:2cm)+.111*(90+60*\a:4cm) + .333*(150 + 60*\a:1cm) $)--($.111*(-30+60*\a+180:2cm)+.111*(30+60*\a+180:4cm) + .333*(-30 + 60*\a+180:1cm) $) ;
\draw[thick] ($.111*(150+60*\a:2cm)+.111*(90+60*\a:4cm) + .333*(150 + 60*\a:2cm) + .333*(150 + 60*\a:1cm) $) -- ($.111*(-30+60*\a+180:2cm)+.111*(30+60*\a+180:4cm) + .333*(-30 + 60*\a+180:2cm) + .333*(-30 + 60*\a+180:1cm) $);
}
\end{tikzpicture}
\caption{The labeling $B^2(T_0)$ (left) and the graph isomorphism to $\face_2$}
\label{labeledbarycentric}
\end{figure}
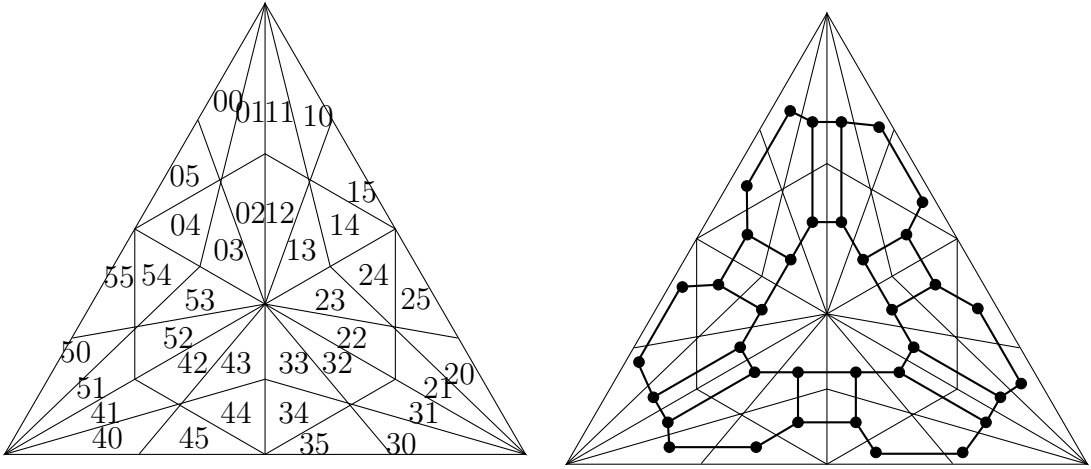

\begin{thm}\label{thm-4-7}
Let $B^n(T_0)$ be the vertex set of a graph where two level $n$ offspring of $T_0$ are connected by an edge if and only if they are adjacent as level $n$ offspring of $T_0$.  This graph is isomorphic to $\face_n$ with the isomorphism given by Proposition~\ref{prop-4-3}.
\end{thm}

\begin{proof}
We already have that $\Phi_n:B^n(T_0)\to\alb^n$ is a bijection between the vertex sets of the two graphs.  It remains to show that $\Phi_n$ preserves the adjacency structure of the two graphs.  Most of the hard work was done in Lemma~\ref{lem-4-6}.  We must now show in particular that if $s,t\in B^n(T_0)$ are adjacent triangles, then $\Phi_n(s)$ and $\Phi_n(t)$ satisfy an edge relation in $\mathsf{F}_k$ for some $1\leq k\leq n$.

Let $s=s_n\subset s_{n-1}\subset\cdots\subset s_1\subset s_0=T_0$ and $t=t_n\subset t_{n-1}\subset\cdots\subset t_1\subset t_0=T_0$ be the family trees for $s$ and $t$, respectively, and let $m$ be maximal with respect to $s_m=t_m$.  Assume $\Phi_m(s_m)=\Phi_m(t_m)=x_1x_2\cdots x_m\in\alb^m$.  First suppose that $m=n-1$ and $\Phi_{n-1}(s_{n-1})=\Phi_{n-1}(t_{n-1})=x_1x_2\cdots x_{n-1}\in\alb^{n-1}$. Then, by part (1) of Lemma~\ref{lem-4-6}, without loss of generality we have $\Phi_n(s)=x_1x_2\cdots x_{n-1}i$ and $\Phi_n(t)=x_1x_2\cdots x_{n-1}j$ where $j=i+1\mod6$.  Thus, $\set{\Phi_n(s),\Phi_n(t))}\in \mathsf{F}_1$, as desired.

Now suppose that $m=n-2$ and $\Phi_{n-2}(s_{n-2})=\Phi_{n-2}(s_{n-2})=x\in\alb^{n-2}$.  We apply part (1) of Lemma~\ref{lem-4-6} to $s_{n-1}$ and $t_{n-1}$ to obtain $\Phi_{n-1}(s_{n-1})=xi$ and $\Phi_{n-1}(t_{n-1})=xj$, where $j=i+1\mod6$ and $i$ is either even or odd.  We apply part (2) of Lemma~\ref{lem-4-6} to $s_n$ and $t_n$ to see that $\Phi_n(s)=x_1x_2\cdots x_{n-2}i\alpha$ and $\Phi_n(t)=x_1x_2\cdots x_{n-2}j\alpha$, where $j=i+1\mod6$, $\alpha$ is the same in the addresses of both $s$ and $t$, and $\alpha$ as above.  Thus, $\set{\Phi_n(s),\Phi_n(t))}\in \mathsf{F}_n$, as desired.

Finally, suppose that $m\leq n-3$.  From parts (1) and (2) of Lemma~\ref{lem-4-6}, we know that $\Phi_{m+2}(s_{m+2})=xi\alpha$ and $\Phi_{m+2}(t_{m+2})=xj\alpha$, where $\Phi_m(s_m)=\Phi_m(t_m)=x\in\alb^m$, $j=i+1\mod6$, and $\alpha$ is as above.  For $3\leq k\leq n-m$, we see that $s_{m+k}$ and $t_{m+k}$ satisfy the conditions of part (3) of Lemma~\ref{lem-4-6}. Thus, the last label in the addresses of both $s_{m+k}$ and $t_{m+k}$ is either $0$ or $5$.  Inductively we have $\Phi_n(s)=xi\alpha v$ and $\Phi_n(t)=xj\alpha v$, where $j=i+1\mod6$, $\alpha$ is as above, and $v\in\set{0,5}^{n-m-2}$.  Thus, $\set{\Phi_n(s),\Phi_n(t))}\in \mathsf{F}_{m+2}$, as desired.

We now see that to each edge in $B^n(T_0)$ there corresponds an edge in $\edge_n$.  We verify from Propositions~\ref{prop-2-5} and~\ref{prop-4-1} that the number of edges in each graph is the same.  Therefore, we have an isomorphism of graphs given by $\Phi_n$.
\end{proof}
\section{Graph Distances\red{\ (Hugo et al.)}}\label{sec:graph_distances}

Each $G_{n}$ inherits a natural planar embedding from $B^{n}(T_{0})$, see figure \ref{labeledbarycentric}. Two interesting features of these embeddings are their central hole and outer border. Figure \ref{level4graph}, generated by the computer program Mathematica, shows a deformed  embedding, accentuating these features. In this section we use the isomorphism established in Theorem~\ref{thm-4-7} to derive formulas for the length of paths that follow the outer border and the central hole. We call these paths outer and inner circumference paths.

\begin{figure}[ht]
\begin{center}
\subfloat{\includegraphics[scale=.35]{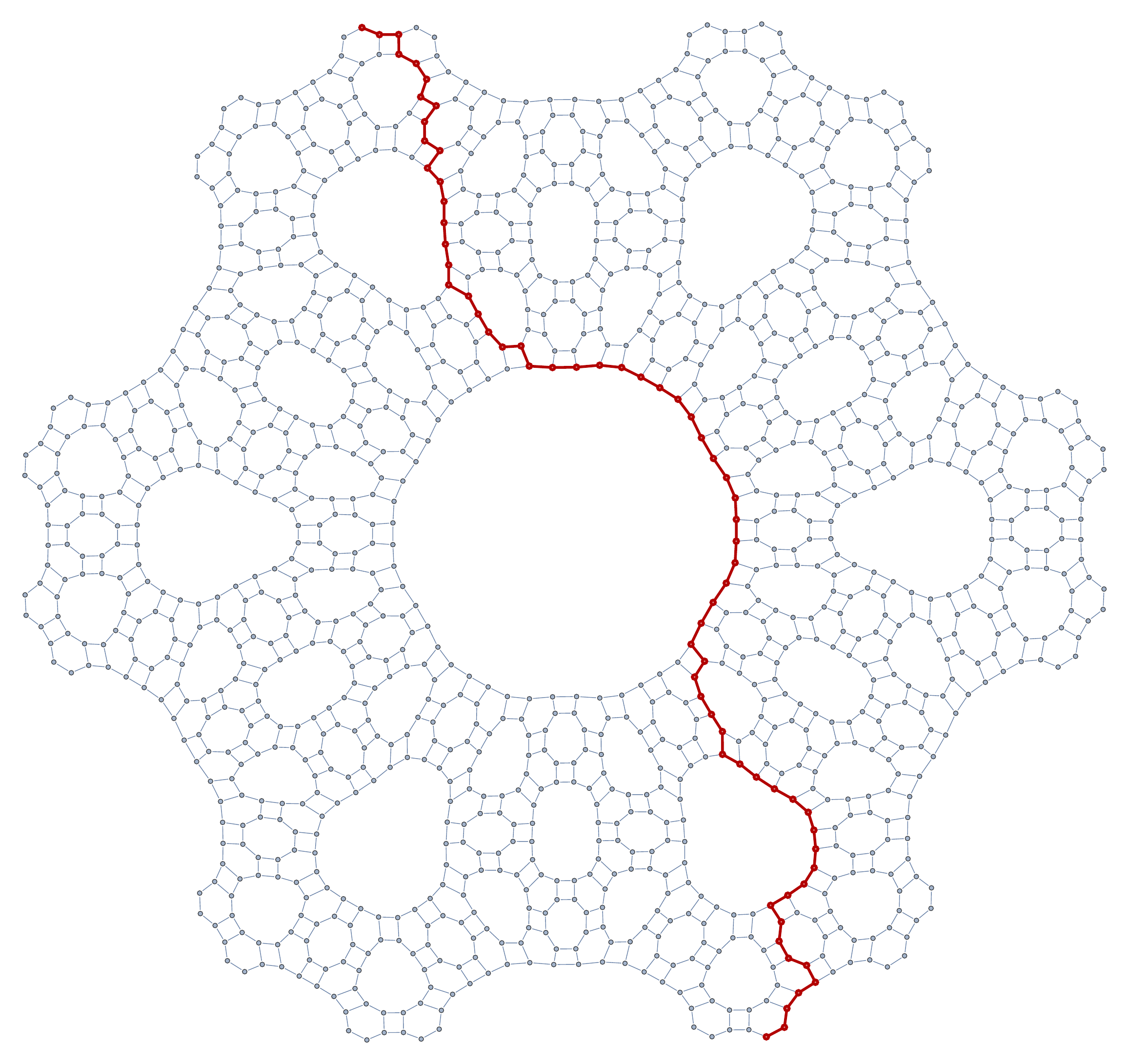}}
\subfloat{\includegraphics[scale=.35]{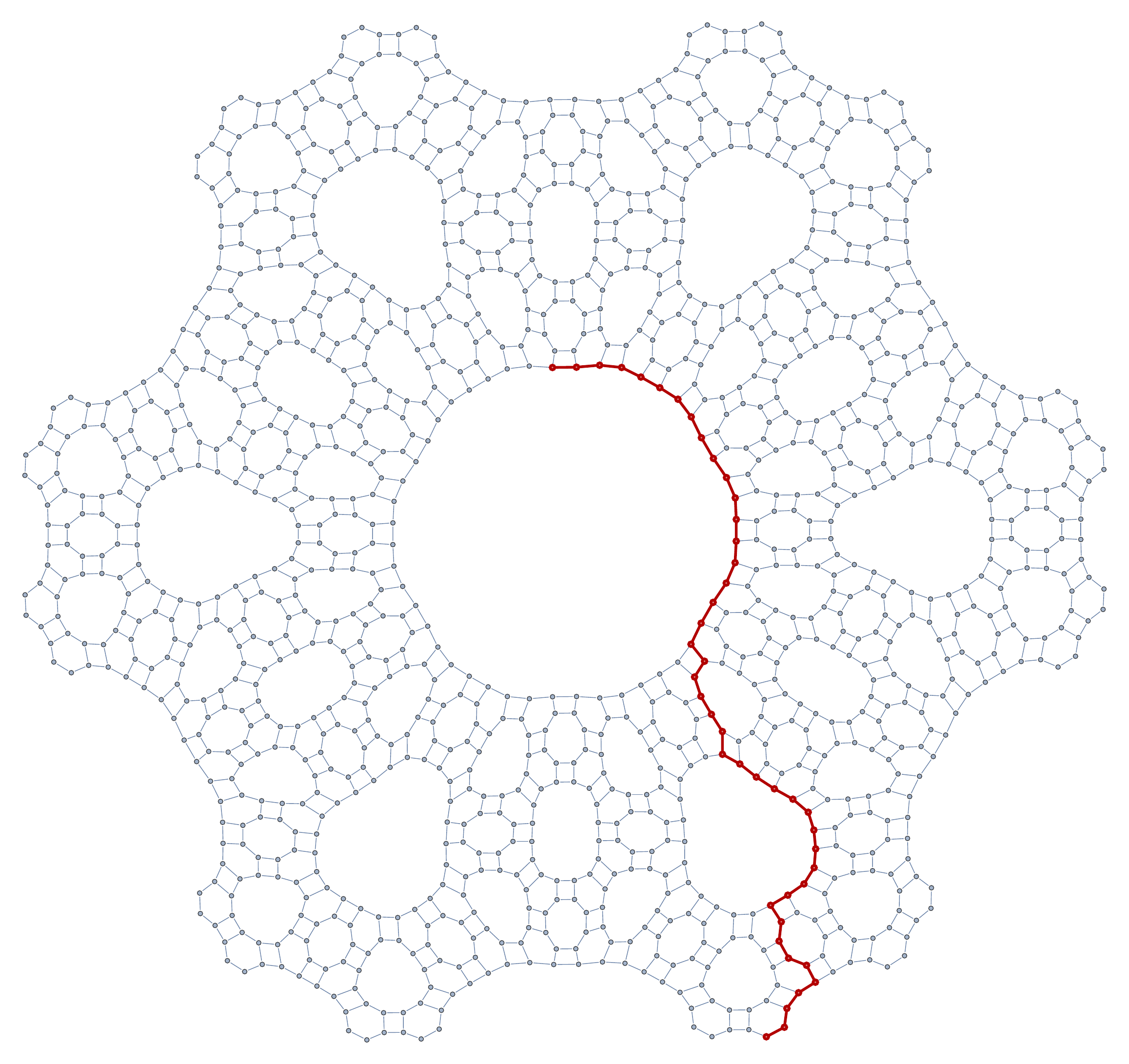}}
\end{center}
\caption{Graph $G_{4}$, with  typical radius (left) and diameter (right) paths highlighted.}
\label{level4graph}
\end{figure}

\begin{defn}\label{defn-5-1}
Let $Y_n$ denote the collection of level $n$ boundary triangles of $T_0$ defined in Definition~\ref{defn-2-1}.  Define the \emph{extended level $n$ boundary triangles of $T_0$} as the set of all $t\in B^n(T_0)$ such that $t$ intersects $\partial T_0$ in exactly one point.  We denote the collection of extended level $n$ boundary triangles of $T_0$ by $Z_n$.  Define the \emph{outer circumference path} of $T_0$, $Out_n$, to be the cycle that crosses each triangle in $Y_n$ and $Z_n$.  Define the \emph{inner circumference path} of $T_0$, $Inn_n$, to be the cycle that crosses each triangle containing the barycenter $b$ of $T_0$ as a vertex.
\end{defn}

\begin{prop}\label{prop-5-2}We have the following formulas for the length of $Out_{n}$ and $Inn_{n}$:
\begin{enumerate}
\item $|Out_{n}|=3n\cdot2^n$
\item $|Inn_{n}|=3\cdot2^n$
\end{enumerate}
\end{prop}

\begin{proof}
We count inductively the length of the outer circumference of $T_0$ and inner circumference of $T_0$.  The statements about $\face_n$ follow from the isomorphism.  When we subdivide any triangle, exactly two of its children contain any one vertex of the parent.  In particular, we have that for each level $n$ triangle in $Z_n$ exactly two of its children are in $Z_{n+1}$.  Similarly, for each level $n$ triangle containing $b$ exactly two of its children contain $b$.  Since $|Inn_1|=6$, it follows inductively that $|Inn_n|=6\cdot2^{n-1}=3\cdot2^n$.

Note that when we subdivide any triangle in $Y_n$, exactly two of its children are in $Y_{n+1}$ and exactly two more of its children are in $Z_{n+1}$.  From Proposition~\ref{prop-2-5}, we know that $|Y_n|=6\cdot2^{n-1}=3\cdot2^n$, thus
\[
|Z_n| = 2|Y_{n-1}|+2|Z_{n-1}| = \sum_{j=1}^{n-1} 2^j|Y_{j-1}| = (n-1)3\cdot2^n. 
\]
So $|Out_n|=|Y_n|+|Z_n|=3n\cdot2^n$.
\end{proof}

Next we recall some standard definitions regarding distance on an arbitrary finite graph $G=(V,E)$, see \cite{Diestel,Foulds} for references. 

\begin{defn}\label{defn-5-3}
The \emph{graph} or \emph{geodesic distance} $d(x,y)$ between two vertices $x,y\in V$ is the length of the shortest edge path connecting them. The \emph{eccentricity} $\mathcal{E}(x)$ of a vertex $x\in V$ is defined as $\mathcal{E}(x):=\max\{d(x,y):y\in V\}$. Now the \emph{diameter} and \emph{radius} of a finite graph G can be defined as $\mathcal{D}(G):=\max\{\mathcal{E}(x):x\in V\}$ and $\mathcal{R}(G):=\min\{\mathcal{E}(x):x\in V\}$, respectively. A vertex $x\in V$ is called \emph{central} if $\mathcal{E}(x)=\mathcal{R}(G)$ and \emph{peripheral} if $\mathcal{E}(x)=\mathcal{D}(G)$. If the length of the shortest edge path connecting a central vertex to another vertex equals the radius of a graph, then we call that path a \emph{radius path}. A \emph{diameter path} is defined analogously. Note that a radius or diameter path connecting two vertices need not be unique. 
\end{defn}

Using Mathematica's graph utilities package we were able to compute the radius and diameter of $G_{n}$ for $1\leq n \leq9$ and to plot radius and diameter paths. See figure 5.1 for some examples. On observing these paths we noticed that the typical radius path of $G_{n+1}$ is composed of a partial path along $Inn_{n+1}$ and an approximate diameter path of $G_{n}$. This suggests the equation
\begin{equation}
\mathcal{R}(G_{n+1})=|PInn_{n+1}|+\mathcal{D}(G_n)-Dadj_{n}
\end{equation}
where $|PInn_{n+1}|$ is the length of the partial path along $Inn_{n+1}$ and $Dadj_{n}$ is the adjustment needed for agreement with our data. Similarly, the well known fact that $\mathcal{R}(G)\leq \mathcal{D}(G) \leq 2\mathcal{R}(G)$ for any finite graph $G$ suggests the equation
\begin{equation}
\mathcal{D}(G_{n})=2\mathcal{R}(G_n)-Radj_{n}
\end{equation}
where $Radj_{n}$ is the adjustment needed for agreement with our data. Both adjustments and $|PInn_{n+1}|$ appear to obey simple recurrence relations, see tables 5.1 and 5.2. Solving these relations using the standard techniques gives the likely formulas 
\begin{equation}
|PInn_{n+1}|=2^{n+1}+1,
\end{equation}
\begin{equation}
Dadj_{n}=\frac{1}{6}(2^{n}-(-1)^{n}-3),
\end{equation}
\begin{equation}
Radj_{n}=\frac{1}{6}(7\cdot2^{n}+2(-1)^{n}+6).
\end{equation}
Combining equations (5.1) through (5.5) yields the recurrence relation
\begin{equation}
\mathcal{R}(G_{n+1})=2\mathcal{R}(G_{n})+\frac{1}{6}(2^{n+2}-(-1)^{n}+3).
\end{equation}
Solving this recurrence relation results in explicit formulas for $\mathcal{R}(G_{n})$ and $\mathcal{D}(G_{n})$.

\begin{conj}\label{conj-5.4} We conjecture the following formulas for the radius and diameter of $G_{n}$:
\begin{enumerate}
\item $\mathcal{R}(G_{n})=\frac{1}{18}(2^{n+1}(13+3n)+(-1)^{n}-9)$  
\item $\mathcal{D}(G_{n})=\frac{1}{9}(2^{n-1}(31+12n)+2(-1)^{n-1}-18)$
\end{enumerate}
\end{conj}

\begin{rem}\label{conj-5.5} If conjecture 5.4 is true then the following are easily seen to hold:
\begin{enumerate}
\item As $n\to\infty$, $|Inn_{n}|=o(\mathcal{D}(G_{n}))$ yet $|Out_{n}|=O(\mathcal{D}(G_{n}))$
\item $\displaystyle\lim_{n\to\infty}\frac{|Out_{n}|}{\mathcal{D}(G_{n})}=\frac{9}{2}$
\item $\mathcal{R}(G_{n+2})=4\mathcal{R}(G_{n+1})-4\mathcal{R}(G_{n})-\frac{1}{2}(1-(-1)^{n})$
\item $\mathcal{D}(G_{n+2})=4\mathcal{D}(G_{n+1})-4\mathcal{D}(G_{n})-2(1+(-1)^{n})$
\end{enumerate}
\end{rem}

\begin{table}[ht]
\begin{center}
\begin{tabular}{|c|*{7}{c}|}
\hline
$n$ & $\mathcal{R}(G_{n+1})$ & = & $|PInn_{n+1}|$ & + & $\mathcal{D}(G_n)$ & - & $Dadj_{n}$ \\
\hline
1 & 8 & = & 5 & + & 3 & - & 0 \\
2 & 19 & = & 9 & + & 10 & - & 0 \\
3 & 44 & = & 17 & + & 28 & - & 1 \\
4 & 99 & = & 33 & + & 68 & - & 2 \\
5 & 220 & = & 65 & + & 160 & - & 5 \\
6 & 483 & = & 129 & + & 364 & - & 10 \\
7 & 1052 & = & 257 & + & 816 & - & 21 \\
8 & 2275 & = & 513 & + & 1804 & - & 42 \\
\vdots & \vdots & = & \vdots & + & \vdots & - & \vdots \\
$n$ & $\mathcal{R}(G_{n+1})$ & = & $2^{n+1}+1$ & + & $\mathcal{D}(G_n)$ & - & $\frac{1}{6}(2^{n}-(-1)^{n}-3)$ \\
\hline
\end{tabular}
\end{center}
\caption{Observed relation between $\mathcal{R}(G_{n+1})$, $|PInn_{n+1}|$, and $\mathcal{D}(G_n)$.}
\end{table}

\begin{table}[ht]
\begin{center}
\begin{tabular}{|c|*{5}{c}|}
\hline
$n$ & $\mathcal{D}(G_n)$ & = & $2\mathcal{R}(G_{n})$ & - & $Radj_{n}$ \\
\hline
1 & 3 & = & 6 & - & 3 \\
2 & 10 & = & 16 & - & 6 \\
3 & 28 & = & 38 & - & 10 \\
4 & 68 & = & 88 & - & 20 \\
5 & 160 & = & 198 & - & 38 \\
6 & 364 & = & 440 & - & 76 \\
7 & 816 & = & 966 & - & 150 \\
8 & 1804 & = & 2104 & - & 300 \\
9 & 3952 & = & 4550 & - & 598 \\
\vdots & \vdots & = & \vdots & - & \vdots \\
$n$ & $\mathcal{D}(G_n)$ & = & $2\mathcal{R}(G_{n})$ & - & $\frac{1}{6}(7\cdot2^{n}+2(-1)^{n}+6)$ \\
\hline
\end{tabular}
\end{center}
\caption{Observed relation between $\mathcal{D}(G_n)$ and $\mathcal{R}(G_{n})$.}
\end{table}

\section{Numerical Data, Concluding remarks and conjectures}\label{sec:numerics}
\red{\ (Matt, et al.)}

From Theorem \ref{thm-4-7}, the level $n$  hexacarpet $G_n$ and the $n$th Barycentric offspring of a triangle are isomorphic.  The following results assume that we are working with vertices from $G_n$, but without loss of generality, we can assume they are cells of $B^n(T_0)$.

\begin{figure}[!]
\begin{center}
\includegraphics[scale=.9]{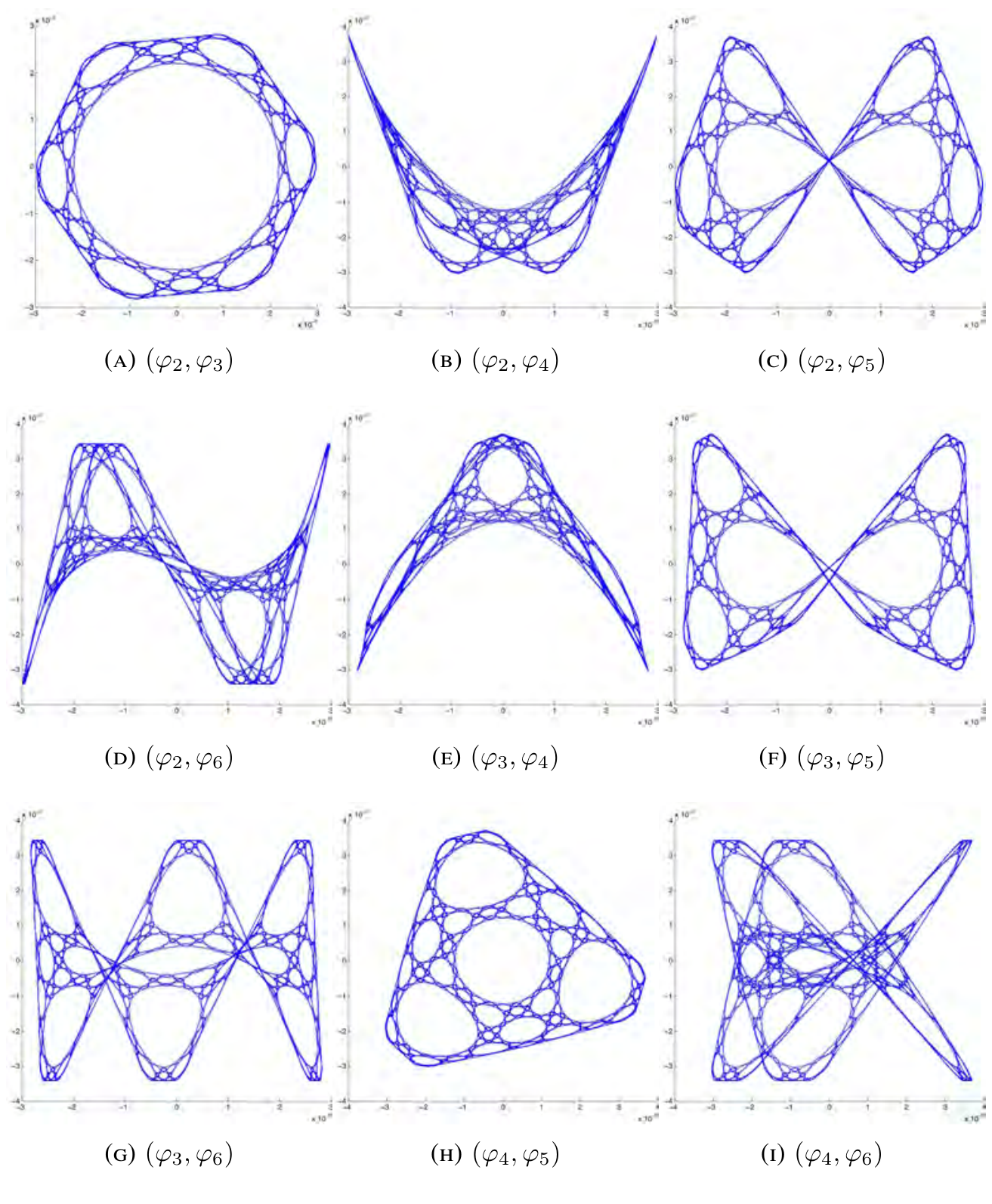}
\end{center}\caption{Two-dimensional eigenfunction coordinates $\big(\phi_i(x),\phi_j(x)\big)$ of the hexacarpet.}\label{eigfpics}
\end{figure}

\begin{figure}[!]
\begin{center}
\includegraphics[scale=.9]{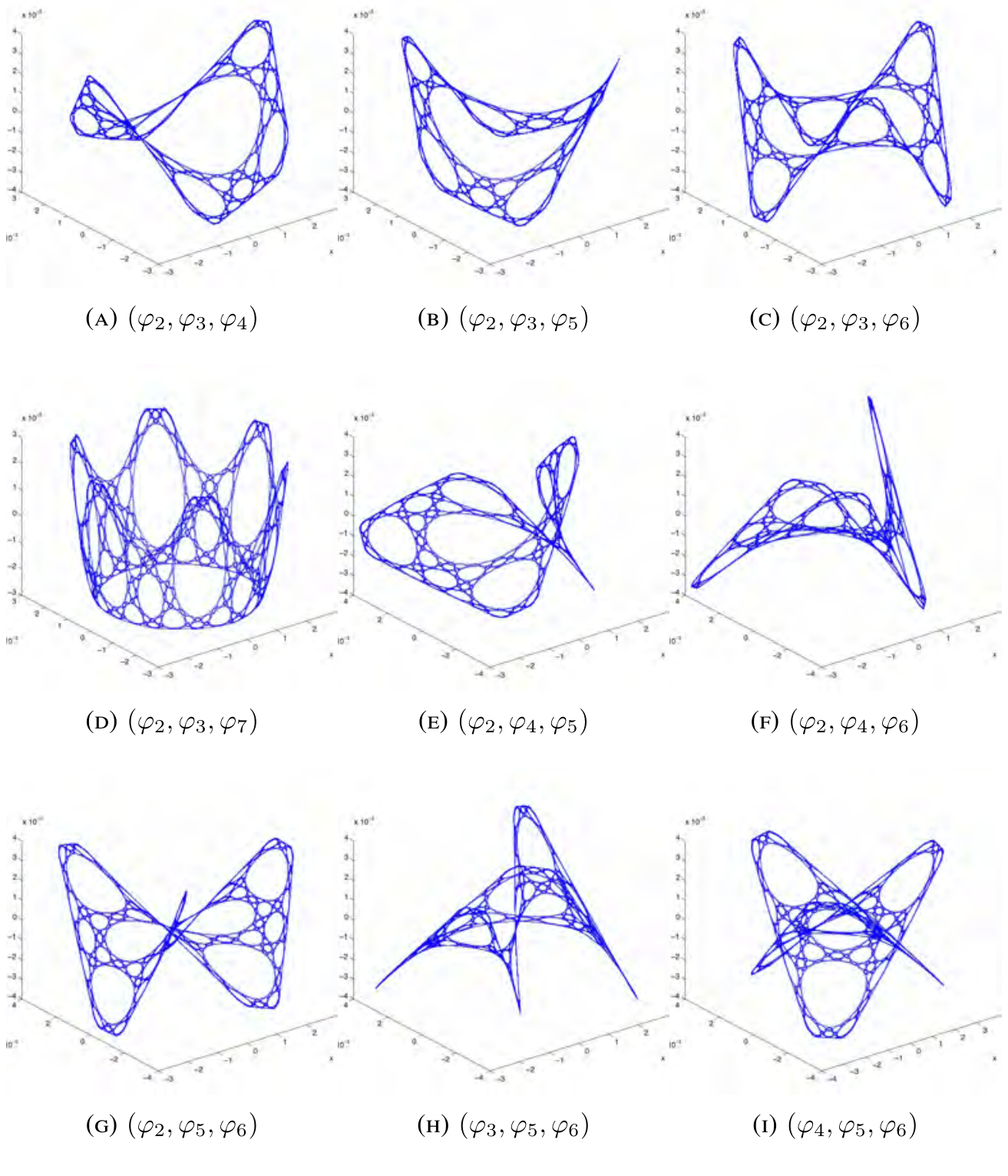}
\end{center}
\caption{Three-dimensional eigenfunction coordinates of the hexacarpet.}\label{3Deigfpics}
\end{figure}

\begin{figure}[!]
\begin{center}
\includegraphics[width=4in,trim=0pt 12pt 0pt 0pt, clip=true]{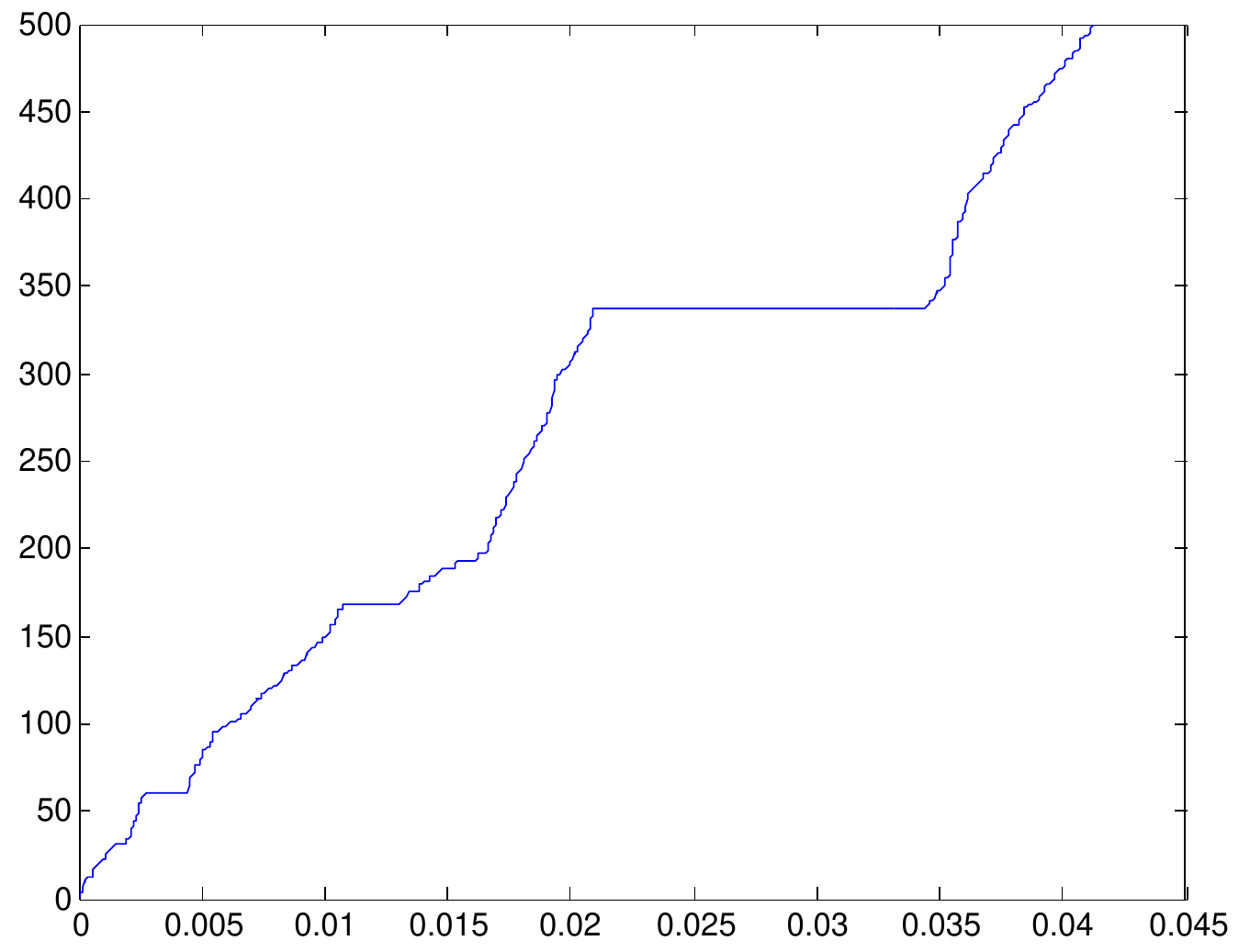}
\end{center}

\

\caption{Shape of the eigenvalue counting function $N(\lambda)$, which is the number of eigenvalues less than $\lambda$, of the 6th level graph for the first 500 eigenvalues. Flat intervals correspond to the gaps in the spectrum. }\label{eigcount}
\end{figure}

We look to solve the eigenvalue on the hexacarpet on $G_n$
\begin{equation}
-\Delta_n u(x)=\lambda u(x)
\end{equation}
at every vertex in $G_n$.  For a finite graph, $G_n$, the graph Laplacian $-\Delta_nu(x)$ is given as 
\begin{equation}
-\Delta_nu(x)=\displaystyle\sum_{\stackrel{x\sim y}{_n} }(u(x)-u(y))
\end{equation}
for every vertex that neighbors $x$ on $G_n$.  Thus for every $6^n$ vertex $x$ in $G_n$ there is a linear equation for $-\Delta_nu(x)$; these equations can be stored into a $6^n$ square matrix.  The eigenvalues and eigenvectors of these graph Laplacian matrices are calculated using the \emph{eigs} command in MATLAB.  Table \ref{eigtable} lists the first twenty eigenvalues for the level-$n$ hexacarpet for $n=7,8$.

\begin{table}[!]
\begin{center}
\scalebox{0.8}{\begin{tabular}{|r||l|l|}
\hline
$\lambda_j$ & $n=7$ & $n=8$ \\ \hline \hline
1&   0.0000   & 0.0000  \\ \hline
2&   1.0000    &1.0000  \\ \hline
3&   1.0000    &1.0000  \\ \hline
4&    3.2798    &3.2798  \\ \hline
5&    3.2798    &3.2798  \\ \hline
6&    5.2033    &5.2032  \\ \hline
7&    7.8389    &7.8386  \\ \hline
8&    7.8389    &7.8386  \\ \hline
9&    8.9141    &8.9139  \\ \hline
10&    8.9141   & 8.9139  \\ \hline
11&    9.4951    &9.4950  \\ \hline
12&    9.4952    &9.4950  \\ \hline
13&   17.5332   &17.5326  \\ \hline
14&   17.5332   &17.5327  \\ \hline
15&   17.6373   &17.6366  \\ \hline
16&   17.6373   &17.6366  \\ \hline
17&   19.8610   &19.8607  \\ \hline
18&   21.7893   &21.7882  \\ \hline
19&   25.7111   &25.7089  \\ \hline
20&   25.7112   &25.7091 \\ \hline
 \end{tabular}} \end{center}\caption{Hexacarpet renormalized eigenvalues at levels $n=7$ and $n=8$.}\label{eigtable}
 \end{table}

\begin{table}[!]
\begin{center}
\scalebox{0.8}{
\begin{tabular}{ | r || l |l |l |l |l |l |l |l  }
\hline
&\multicolumn{7}{|c|}{Level $n$} \\ \hline
$\rho$ & 1 & 2& 3& 4& 5 & 6& 7 \\ \hline \hline  
1& &&&&&&  \\ \hline 
2&    1.2801   &   1.3086   &   1.3085   &   1.3069   &   1.3067   &   1.3065   &   1.3064 \\ \hline
3 &   1.2801   &   1.3086   &   1.3079   &   1.3075   &   1.3066   &   1.3065   &   1.3064 \\ \hline
4 &   1.1761   &   1.3011   &   1.3105   &   1.3064   &   1.3068   &   1.3065   &   1.3065 \\ \hline
5 &   1.1761   &   1.3011   &   1.3089   &   1.3074   &   1.3073   &   1.3065   &   1.3065 \\ \hline
6 &   1.0146   &   1.2732   &   1.3098   &   1.3015   &   1.3067   &   1.3065   &   1.3064 \\ \hline
7      &   &   1.2801   &   1.3114   &   1.3055   &   1.3071   &   1.3066   &   1.3065 \\ \hline
8         &       &   1.2801   &   1.3079   &   1.3086   &   1.3075   &   1.3067   &   1.3065 \\ \hline
9         &       &   1.2542   &   1.3191   &   1.2929   &   1.3056   &   1.3065   &   1.3065 \\ \hline
10      &       &   1.2542   &   1.3017   &   1.3089   &   1.3069   &   1.3066   &   1.3065 \\ \hline
 11      &      &   1.2461   &   1.3051   &   1.3063   &   1.3048   &   1.3065   &   1.3065 \\ \hline
 12       &   &   1.2461   &   1.3019   &   1.3075   &   1.3068   &   1.3066   &   1.3065 \\ \hline
 13      &      &   1.1969   &   1.6014   &   1.0590   &   1.3068   &   1.3066   &   1.3065 \\ \hline
 14      &      &   1.1969   &   1.2972   &   1.3063   &   1.3078   &   1.3066   &   1.3065 \\ \hline
  15&         &   1.2026   &   1.3059   &   1.3020   &   1.3060   &   1.3066   &   1.3065 \\ \hline
 16      &      &   1.2026   &   1.2993   &   1.3074   &   1.3071   &   1.3067   &   1.3065 \\ \hline
 17      &     &   1.1640   &   1.3655   &   1.2349   &   1.3064   &   1.3066   &   1.3065 \\ \hline
 18      &      &   1.1755   &   1.4128   &   1.2009   &   1.3069   &   1.3067   &   1.3065 \\ \hline
 19      &      &   1.1761   &   1.5252   &   1.1171   &   1.3073   &   1.3068   &   1.3066 \\ \hline
 20      &      &   1.1761   &   1.2988   &   1.3114   &   1.3077   &   1.3068   &   1.3065 \\ \hline
\end{tabular}}
\end{center}
\caption{Hexacarpet estimates for resistance coefficient $\rho$ given by $\frac{1}{6} \frac{\lambda_j^n}{\lambda_j^{n+1}}$.}\label{ctable}
\end{table}

Using our data, we are  able to plot the hexacarpet in eigenvalue coordinates.  This is analogous to harmonic coordinates which are used extensively  in the literature on fractals, see \cite{Kajino2012,Ki08,teply} and references therein.  Given two eigenfunctions, $\varphi, \psi$, defined on $G_n$, we plot the ordered pair $(\varphi(x),\psi(x))$ for each $x\in G_n$.  Figure \ref{eigfpics} shows the plots of $(\varphi_i,\varphi_j)$ for $2\leq i< j \leq 6$ for level $n=7$.  The first eigenfunction $\phi_1$ is a constant function at each level and is excluded.  Figure~\ref{3Deigfpics} shows  three dimensional plots $(\varphi_i,\varphi_j, \varphi_k)$ for some choices of $i,j,k$. These plots demonstrate that the geometry of the hexacarpet is significantly different from the geometry of the triangle, on which the hexacarpet construction is based.

In \cite{outer}, it is believed that there is a renormalization factor $\tau$ such that
\begin{equation}
\lambda_j^n=\tau \lambda_j^{n+1}
\end{equation}
where $\lambda_j^n$ represents the $j$th eigenvalue of the level $n$ graph Laplacian.  According to \cite{outer}, this coefficient $\tau$ is the Laplacian scaling factor (it is denoted $R$ in \cite{outer}). This scaling factor  satisfies $\tau=N\rho$, where $N$ is the factor that describes how $X^n$ grows at each level $n$ or,  alternatively, $N$ can also be thought of as the number of contraction maps for a carpet, and $\rho$ is the resistance scaling coefficient (see for instance \cite{Ba,kig01,Ngask} and references therein).  In the case of the hexacarpet, $N=6$ and $\rho$ is not known (in the present state of knowledge it is not possible to compute $\rho$ theoretically for the hexacarpet). 
Since the renormalized ratios of eigenvalues seems to converge, as seen in table \ref{ctable}, we venture Conjecture~\ref{conj-1} part (1) and (2) in the possibility that a Laplacian and its corresponding diffusion process exist on the limit object, the hexacarpet.  The best estimate for $\rho$ will come from our estimates for the lowest nonzero eigenvalue, thus giving us an estimate of $\rho\approx 1.3064$. The hexacarpet has six contraction mappings, and the natural contraction ratio from Section~\ref{sec:bcs} is $1/2$. With $\tau = 6 \rho$, our calculation above suggests that the spectral dimension is $2\log(6)/\log(\tau) \approx 1.74$. Part (3) of Conjecture~\ref{conj-1} comes from the fact that the hexacarpet has well defined reflection symmetries, which make it plausible to apply the methods of \cite{BBK,BBKT}. 
Moreover, we expect logarithmic corrections because the approximating graphs have diameters which seem to grow on the order of $n6^n$ rather than $6^n$, as in Proposition~\ref{prop-5-2} and Conjecture~\ref{conj-5.4}.
There is an extensive literature on heat kernel estimates and their relation to functional inequalities (see 
\cite{BBK,BBKT,GT2001,GT2002,GT2011,Kig-HKE,St-function} and references therein), which provides general framework and background for our calculations and conjectures.

The reasoning for part (4) of the conjecture is illustrated in Figure~\ref{eigcount}, as gaps in the spectrum correspond to flat intervals in the eigenvalue counting function. It is a worthwhile point of investigation because of  \cite{St-gaps,HSTZ}. 
Finally, part (5) of the conjecture is ventured because of recent work on spectral zeta functions \cite{ST} (see \cite{LvF,T-zeta} for background).



\bibliographystyle{amsalpha}
\def\cprime{$'$} \def\cprime{$'$}
\providecommand{\bysame}{\leavevmode\hbox to3em{\hrulefill}\thinspace}
\providecommand{\MR}{\relax\ifhmode\unskip\space\fi MR }
\providecommand{\MRhref}[2]{%
  \href{http://www.ams.org/mathscinet-getitem?mr=#1}{#2}
}
\providecommand{\href}[2]{#2}

%
%
\end{document}